\documentclass[a4paper]{mcom-l}
\usepackage{amsmath,amsthm,amsfonts,amssymb}
\usepackage{mathtools}
\usepackage{enumitem}\setitemize{leftmargin=5mm}
\usepackage{accents}
\usepackage{color}
\usepackage[usenames,dvipsnames,svgnames,table]{xcolor}
\usepackage{cite}
\usepackage{soul}
\usepackage{hyperref}
\usepackage[capitalize]{cleveref}
\usepackage{todonotes}
\usepackage{fancyhdr}
\usepackage{xfrac}

\usepackage{verbatim}

\usepackage{tikz}
\usepackage{pgfplots,pgfplotstable}
\usepgfplotslibrary{colorbrewer,groupplots}
\usepackage{caption}

\pgfplotsset{
cycle list/Set1-5,
cycle multiindex* list={
mark list*\nextlist
Set1-5\nextlist
},
}

\usepackage{stmaryrd} %
\usepackage{pgf,tikz}
\usetikzlibrary{spy,matrix, calc, arrows,shapes,decorations}
\usepackage{tkz-euclide}
\usetikzlibrary{positioning,arrows,calc,decorations.markings,math,arrows.meta}
\usepackage{pgfplots} 
\usepgfplotslibrary{groupplots}
\usepackage{csvsimple} 
\usepackage{siunitx} %
\usepackage{tabularx}
\usepackage{booktabs}
\usepackage{multirow}
\usepackage[normalem]{ulem}
\useunder{\uline}{\ul}{}

\pgfplotscreateplotcyclelist{unfmixed1}{%
 teal,every mark/.append style={solid,fill=teal!80!black},mark=*,mark size=2.5pt\\
 red!70!white,every mark/.append style={solid,fill=red!80!black},mark=pentagon*,mark size=2.5pt\\
 ForestGreen,every mark/.append style={solid,fill=ForestGreen!80!black},mark=triangle*,mark size=2.5pt\\
cyan!60!black,every mark/.append style={solid,fill=cyan!80!black},mark=diamond*,mark size=2.5pt\\
orange,every mark/.append style={solid,fill=orange!80!black},mark=square*,mark size=2.5pt\\
}

\pgfplotscreateplotcyclelist{unfmixed2}{%
teal,every mark/.append style={solid,fill=teal!80!black},mark=*,mark size=2.5pt\\
orange,every mark/.append style={solid,fill=orange!80!black},mark=square*,mark size=2.5pt\\
}

\pgfplotscreateplotcyclelist{unfmixed3}{%
 orange,every mark/.append style={solid,fill=orange!80!black},mark=square*,mark size=2.5pt\\
 teal,every mark/.append style={solid,fill=teal!80!black},mark=*,mark size=2.5pt\\
 cyan!60!black,every mark/.append style={solid,fill=cyan!80!black},mark=diamond*,mark size=2.5pt\\
 red!70!white,every mark/.append style={solid,fill=red!80!black},mark=pentagon*,mark size=2.5pt\\
 ForestGreen,every mark/.append style={solid,fill=ForestGreen!80!black},mark=triangle*,mark size=2.5pt\\
}

\pgfplotscreateplotcyclelist{paulcolors}{%
violet,every mark/.append style={solid,fill=violet},mark=square*,very thick,mark size=3pt\\%
teal,every mark/.append style={solid,fill=teal},mark=*,very thick,mark size=3pt\\%
orange,every mark/.append style={solid,fill=orange},mark=diamond*,very thick,mark size=3pt\\%
violet,densely dashed,every mark/.append style={solid,fill=violet},mark=square*,very thick,mark size=3pt\\%
teal,densely dashed,every mark/.append style={solid,fill=teal},mark=*,very thick,mark size=3pt\\%
orange,densely dashed,every mark/.append style={solid,fill=orange},mark=diamond*,very thick,mark size=3pt\\%
violet,dash dot,every mark/.append style={solid,fill=violet},mark=square*,very thick,mark size=3pt\\%
teal,dash dot,every mark/.append style={solid,fill=teal},mark=*,very thick,mark size=3pt\\%
orange,dash dot,every mark/.append style={solid,fill=orange},mark=diamond*,very thick,mark size=3pt\\%
}
\pgfplotscreateplotcyclelist{paulcolors2}{%
teal,every mark/.append style={solid,fill=teal},mark=*,very thick,mark size=3pt\\%
orange,every mark/.append style={solid,fill=orange},mark=diamond*,very thick,mark size=3pt\\%
teal,densely dashed,every mark/.append style={solid,fill=teal},mark=*,very thick,mark size=3pt\\%
orange,densely dashed,every mark/.append style={solid,fill=orange},mark=diamond*,very thick,mark size=3pt\\%
teal,dash dot,every mark/.append style={solid,fill=teal},mark=*,very thick,mark size=3pt\\%
orange,dash dot,every mark/.append style={solid,fill=orange},mark=diamond*,very thick,mark size=3pt\\%
}
\pgfplotscreateplotcyclelist{paulcolors1}{%
teal,every mark/.append style={solid,fill=teal},mark=*,very thick,mark size=3pt\\%
}
\pgfplotscreateplotcyclelist{paulcolorsfill}{%
teal,fill=teal,fill opacity=0.5,thick\\
orange,fill=orange,fill opacity=0.5,thick\\
violet,fill=violet,fill opacity=0.5,thick\\
}

\pgfplotsset{
    discard if not/.style 2 args={
        x filter/.append code={
            \edef\tempa{\thisrow{#1}}
            \edef\tempb{#2}
            \ifx\tempa\tempb
            \else
                
            \fi
        }
    }
}
\pgfplotsset{compat=1.16}%

\pgfplotsset{tick label style={font=\small},label style={font=\small},legend style={font=\small},}
\pgfplotsset{ width=.49\linewidth}

\usepackage{listings}
\usepackage[plain]{algorithm}

\definecolor{maroon}{cmyk}{0, 0.87, 0.68, 0.32}
\definecolor{halfgray}{gray}{0.55}
\definecolor{ipython_frame}{RGB}{207, 207, 207}
\definecolor{ipython_bg}{RGB}{247, 247, 247}
\definecolor{ipython_red}{RGB}{186, 33, 33}
\definecolor{ipython_green}{RGB}{0, 128, 0}
\definecolor{ipython_cyan}{RGB}{64, 128, 128}
\definecolor{ipython_purple}{RGB}{170, 34, 255}

\lstdefinelanguage{iPython}{
    morekeywords=[2]{abs,all,any,basestring,bin,bool,bytearray,callable,chr,classmethod,cmp,compile,complex,delattr,dict,dir,divmod,enumerate,eval,execfile,file,filter,float,format,frozenset,getattr,globals,hasattr,hash,help,hex,id,input,int,isinstance,issubclass,iter,len,list,locals,long,map,max,memoryview,min,next,object,oct,open,ord,pow,property,range,raw_input,reduce,reload,repr,reversed,round,set,setattr,slice,sorted,staticmethod,str,sum,super,tuple,type,unichr,unicode,vars,xrange,zip,apply,buffer,coerce,intern},%
    sensitive=true,%
    morecomment=[l]\#,%
    morestring=[b]',%
    morestring=[b]",%
    morestring=[s]{'''}{'''},%
    morestring=[s]{"""}{"""},%
    morestring=[s]{r'}{'},%
    morestring=[s]{r"}{"},%
    morestring=[s]{r'''}{'''},%
    morestring=[s]{r"""}{"""},%
    morestring=[s]{u'}{'},%
    morestring=[s]{u"}{"},%
    morestring=[s]{u'''}{'''},%
    morestring=[s]{u"""}{"""},%
    literate=
    {á}{{\'a}}1 {é}{{\'e}}1 {í}{{\'i}}1 {ó}{{\'o}}1 {ú}{{\'u}}1
    {Á}{{\'A}}1 {É}{{\'E}}1 {Í}{{\'I}}1 {Ó}{{\'O}}1 {Ú}{{\'U}}1
    {à}{{\`a}}1 {è}{{\`e}}1 {ì}{{\`i}}1 {ò}{{\`o}}1 {ù}{{\`u}}1
    {À}{{\`A}}1 {È}{{\'E}}1 {Ì}{{\`I}}1 {Ò}{{\`O}}1 {Ù}{{\`U}}1
    {ä}{{\"a}}1 {ë}{{\"e}}1 {ï}{{\"i}}1 {ö}{{\"o}}1 {ü}{{\"u}}1
    {Ä}{{\"A}}1 {Ë}{{\"E}}1 {Ï}{{\"I}}1 {Ö}{{\"O}}1 {Ü}{{\"U}}1
    {â}{{\^a}}1 {ê}{{\^e}}1 {î}{{\^i}}1 {ô}{{\^o}}1 {û}{{\^u}}1
    {Â}{{\^A}}1 {Ê}{{\^E}}1 {Î}{{\^I}}1 {Ô}{{\^O}}1 {Û}{{\^U}}1
    {œ}{{\oe}}1 {Œ}{{\OE}}1 {æ}{{\ae}}1 {Æ}{{\AE}}1 {ß}{{\ss}}1
    {ç}{{\c c}}1 {Ç}{{\c C}}1 {ø}{{\o}}1 {å}{{\r a}}1 {Å}{{\r A}}1
    {€}{{\EUR}}1 {£}{{\pounds}}1
    {^}{{{\color{ipython_purple}\^{}}}}1
    {=}{{{\color{ipython_purple}=}}}1
    {+}{{{\color{ipython_purple}+}}}1
    {*}{{{\color{ipython_purple}$^\ast$}}}1
    {/}{{{\color{ipython_purple}/}}}1
    {+=}{{{+=}}}1
    {-=}{{{-=}}}1
    {*=}{{{$^\ast$=}}}1
    {/=}{{{/=}}}1,
    literate=
    *{-}{{{\color{ipython_purple}-}}}1
     {?}{{{\color{ipython_purple}?}}}1,
    identifierstyle=\color{black}\ttfamily,
    commentstyle=\color{ipython_cyan}\ttfamily,
    stringstyle=\color{ipython_red}\ttfamily,
    keepspaces=true,
    showspaces=false,
    showstringspaces=false,
    rulecolor=\color{ipython_frame},
    framexleftmargin=0mm,
    numbers=left,
    numberstyle=\tiny\color{halfgray},
    numbersep=1mm,
    xleftmargin=1mm,
    basicstyle=\scriptsize,
    keywordstyle=\color{ipython_green}\ttfamily,
}

\lstdefinestyle{trefftzy}{
    language=iPython,
    emptylines=1,
    breaklines=true,
    basicstyle=\footnotesize\ttfamily\color{black},    
    moredelim=**[is][\color{teal}]{<}{>},
    moredelim=**[is][\color{purple}]{'}{'},
}

\theoremstyle{plain}
\newtheorem{theorem}{Theorem}%
\newtheorem{lemma}[theorem]{Lemma}

\newtheorem{corollary}[theorem]{Corollary}

\theoremstyle{definition}

\newtheorem{remark}{Remark}
\newtheorem{assumption}{Assumption}%

\newcommand{\bpm}{\begin{pmatrix}}
\newcommand{\epm}{\end{pmatrix}}

\newcommand{\jump}[1]{[\![ #1 ]\!]}

\newcommand{\Th}{{\mathcal{T}_h}} 
\newcommand{\Fh}{{\mathcal{F}_h}} 

\newcommand{\dom}{\Omega} 
 
\newcommand{\OmT}{\Omega^{\mathcal{T}}}
\newcommand{\Omg}{\Omega_{\gamma}}
\newcommand{\Omi}{\Omega^{\text{int}}}
\newcommand{\OmG}{\Omega^{\Gamma}}
\newcommand{\Thi}{\mathcal{T}_h^i}
\newcommand{\ThG}{\mathcal{T}_h^\Gamma}

\newcommand{\Sz}{\Sigma^0}
\newcommand{\Sh}{\Sigma_h}
\newcommand{\Shz}{\Sigma_h^0}
\newcommand{\Shperp}{\Sigma_h^\perp{\!}}
\newcommand{\Sperp}{\Sigma^\perp{\!}}

\newcommand{\PiRT}{\Pi^{\mathbb{R}\mathbb{T}}}
\newcommand{\ue}{u^{\mathcal{E}}}
\newcommand{\pe}{p^{\mathcal{E}}}
\newcommand{\fe}{f^{\mathcal{E}}}

\renewcommand{\div}{\operatorname{div}}

\DeclareMathOperator{\id}{id}

\DeclareMathOperator{\diam}{diam}

\DeclareMathOperator{\meas}{meas}

\newcommand*{\norm}[1]{\left\|#1\right\|}

\newcommand\restr[2]{{ \left.\kern-\nulldelimiterspace #1 \vphantom{\big|} \right|_{#2} }}

\newcommand{\bn}{\mathbf{n}}

\definecolor{pscol}{rgb}{0.8,0,0}

\definecolor{clcol}{rgb}{0,0.5,0.6}

\newcommand{\nrm}[1]{\left\Vert #1 \right\Vert}

\begin{document}

\title[Analysis of divergence-preserving unfitted FEM]{Analysis of divergence-preserving unfitted finite element methods for the mixed Poisson problem}

\author{Christoph Lehrenfeld}
\address{Institute for Numerical and Applied Mathematics\\
University of Göttingen\\
Lotzestr. 16-18, 37083 Göttingen, Germany\\}
\email{lehrenfeld@math.uni-goettingen.de}

\author{Tim van Beeck}
\address{Institute for Numerical and Applied Mathematics\\
University of Göttingen\\
Lotzestr. 16-18, 37083 Göttingen, Germany\\}
\email{tim.vanbeeck@stud.uni-goettingen.de}

\author{Igor Voulis}
\address{Institute for Numerical and Applied Mathematics\\
University of Göttingen\\
Lotzestr. 16-18, 37083 Göttingen, Germany\\}
\email{i.voulis@math.uni-goettingen.de}

\subjclass[2010]{Primary }

\date{\today}

\dedicatory{}

\begin{abstract}
In this paper we present a new $H(\div)$-conforming unfitted finite element method for the mixed Poisson problem which is robust in the cut configuration and preserves conservation properties of body-fitted finite element methods. The key is to formulate the divergence-constraint on the active mesh, instead of the physical domain, in order to obtain robustness with respect to cut configurations without the need for a stabilization that pollutes the mass balance. This change in the formulation results in a slight inconsistency, but does not affect the accuracy of the flux variable. By applying post-processings for the scalar variable, in virtue of classical local post-processings in body-fitted methods, we retain optimal convergence rates for both variables and even the superconvergence after post-processing of the scalar variable. We present the method and perform a rigorous a-priori error analysis of the method and discuss several variants and extensions. Numerical experiments confirm the theoretical results.
\end{abstract}


\maketitle

\section{Introduction}
In the recent decades unfitted finite element methods have become a popular tool for the numerical approximation of partial differential equations (PDEs) on complex geometries. The separation of the geometry description from the computational mesh allows for a flexible handling of the geometry and avoids meshing and re-meshing issues of body-fitted discretizations.
Unfitted finite element methods go under different names such as 
\emph{CutFEM}~\cite{B15}, \emph{extended FEM} (X-FEM)\cite{BMUP01,MDB99}, \emph{Finite Cell} \cite{PDR07} and others. 
Several issues that come from the construction of the unfitted methods, such as the enforcement of boundary and interface conditions, the numerical integration on cut cells and the conditioning of the resulting linear systems, have been addressed in the literature for a variety of PDE problems.

One crucial development in the field of unfitted finite element methods was the introduction of the \emph{ghost penalty} (GP) stabilization \cite{Bur10} as a means of dealing with locally arbitrary cut configurations. 
In these situations GP stabilization serves two purposes: ensuring stability of an underlying body-fitted discretization in exact arithmetics and robustness of linear systems in terms of conditioning.
In the following, we will distinguish these two aspects as \emph{stability} and \emph{robustness}, respectively.
An alternative approach with the same goals is the concept of cell agglomeration techniques, see e.g. \cite{BVM18}.

Despite these advances, there are still several situations where features of body-fitted discretization do not carry over to the unfitted setting. Mixed finite element formulations, which rely on \emph{compatibility} of the finite element spaces involved, are one such situation. Compatibility here means that, compared to other methods, additional properties of the PDE solution carry over from the continuous to the discrete level. The most prominent example is local conservation.

In this paper we consider the mixed Poisson problem as a model problem for a mixed finite element formulation with additional \emph{compatibility}. We expect that several of the ideas presented here can be extended to other mixed problems, such as the Stokes problem.

Several papers in the literature have considered the use of $H(\div)$-conforming or divergence preserving finite elements in an unfitted setting for the mixed Poisson, the Darcy or the Stokes problem.
For the lowest order Raviart-Thomas elements an unfitted mixed finite element method has been proposed and analyzed for the Darcy interface problem in the 2D case in \cite{d2012mixed}. An immersed finite element method for a similar interface problem has recently been treated in \cite{Ji22}. Both formulations do not require ghost-penalty stabilization, but are restricted to the lowest order case. In \cite{puppi2021cut,caoextended} higher order Raviart-Thomas elements are considered, but a ghost-penalty stabilization is applied which pollutes the conservation properties of the method at least in the vicinity of cut elements. 
In the context of Stokes problems, the stabilization of the pressure variable in unfitted finite element formulations has been considered in \cite{massing2014stabilized,kirchhart2016analysis,guzman2018inf}. However, in these discretizations, the underlying body-fitted Stokes discretization is not locally mass-conservative in the first place.
Recently, \cite{liu2021cutfem} have considered a discretization for Stokes that is exactly divergence-free (and pressure robust) in the body-fitted case. For the unfitted case, they use a pressure stabilization and thus preserve the divergence-free property only in the vicinity of cut elements.

In \cite{burman2022cut}, a low order discretization of the Stokes problem is introduced which imposes the divergence-free constraint on the entire active mesh. Since the underlying velocity-pressure pair is \emph{compatible}, the discrete solution is exactly divergence-free, stable and robust in the cut configuration without additional ghost-penalty-type stabilization that involves the pressure variable.
The use of ghost penalties to stabilize a discretization in a way that does not pollute the conservation balance has recently been achieved in \cite{frachon2022divergence}, where the authors consider a mixed unfitted finite element method for the Darcy problem.

In this work we use similar ideas as in \cite{frachon2022divergence} and \cite{burman2022cut} to achieve stability without the need for ghost penalties or cell agglomeration.

\subsection*{Main contributions}
In this manuscript, we present a new unfitted mixed finite element method with an enforcement of the divergence constraint that avoids the need for ghost penalties or cell agglomeration on the scalar variable.
The resulting scheme introduces an inconsistency. We show that this inconsistency does not affect the accuracy of the flux variable which allows us to apply a post-processing to the scalar variable to obtain optimal convergence rates for both variables and even superconvergence after post-processing of the scalar variable.
We also discuss the use of hybridization and Neumann boundary conditions.

\subsection*{Structure of the paper}
In \cref{sec:preliminaries}, we introduce the problem, i.e. unfitted meshes and the PDE problem, as well as some further preliminary considerations.
The main method is presented in \cref{sec:discretization} and arises from a seemingly inconsistent  modification of a simple, but not stable, unfitted mixed formulation. To deal with the inconsistency, a simple post-processing, similar to that used in usual mixed methods for Poisson or Darcy problems, is applied. 
In \cref{sec:analysis1} the method without post-processing is analysed in terms of a full a priori error analysis. \Cref{sec:post-process} then introduces and analyses two possible post-processing schemes which lead to optimal, i.e., superconvergent approximations of the primal unknown.

Further modifications of the main method with computational advantages or a broader range of applications are given in \cref{sec:further}. Finally, 
\cref{sec:numerics} discusses numerical results and computational aspects of the discussed methods. 

\section{Preliminaries} \label{sec:preliminaries}
\subsection{Unfitted geometry, computational meshes and patches}\label{sec:patches}
We consider a PDE problem posed on an unfitted open bounded domain $\Omega \subset \mathbb{R}^d,~d=2,3$ with Lipschitz boundary $\Gamma := \partial \Omega$. \emph{Unfitted} means that $\Omega$ is not parametrized by the computational mesh that is used for the finite element approximation of the PDE solution.

Let $\widetilde \Omega$ be a background domain, sufficiently large such that $\overline{\Omega} \subseteq \widetilde \Omega$ and let $\widetilde{\mathcal{T}}_h = \{T\}$ be a shape regular, simplicial and quasi-uniform triangulation of $\widetilde\Omega$. 
The set of elements that overlap with $\Omega$ is called \emph{the active mesh}, while the set of elements that intersect with $\partial \Omega$ is called \emph{the cut mesh}:
\begin{equation*}
  \Th = \{ T \in \widetilde{\mathcal{T}}_h \mid \meas_{d}(T \cap \Omega)>0 \},
  \quad
  \ThG  = \{ T \in \widetilde{\mathcal{T}}_h \mid \meas_{d-1}(T \cap \Gamma)>0 \}.
\end{equation*}
We denote the domains corresponding to the active and cut mesh as $\OmT$ and $\OmG$, respectively, and by $\Thi = \Th \setminus \ThG$ we denote the set of interior \emph{uncut} elements with associated domain $\Omi$.
Given a subset $S_h = \{T\} \subseteq \Th$, we denote the set of interior element interfaces as $\Fh(S_h)$ and denote $\Fh = \Fh(\Th)$. A facet patch, consisting of the two adjacent elements of a facet $F$ is denoted as $\omega_F$.

The local mesh size of a mesh element $T\in\widetilde{\mathcal{T}}_h$ is defined as $h_T = \diam(T) \coloneqq \max_{\mathbf{x}_1, \mathbf{x}_2 \in T} \nrm{\mathbf{x}_1 -
  \mathbf{x}_2}_2$ and the global mesh size is $h = \max_{T\in \widetilde{\mathcal{T}}_h} h_T$.

As trimmed cut elements $T \cap \Omega,~T\in\ThG$ may become arbitrary small and may suffer from shape irregularity, we will group together sets of neighboring elements in $\Th$ in the vicinity of $\partial \Omega$ into disjoint patches. Each of these patches $\mathcal{T}_\omega$ with associated domain $\omega$ contains at least one \emph{root element} $T_\omega \in \Thi$. The interior facets of a patch are denoted as $\mathcal{F}_h^\omega:=\Fh(\mathcal{T}_\omega)$. Loosely speaking, these patches will allow us to distribute ``stability'' from interior elements to cut elements (through stabilizations acting on its interior facets). We denote by $\mathcal{C}_h =
\{\mathcal{T}_\omega\}$ the set of disjoint patches and define the set of all patch-interior facets $\mathcal{F}_h^\mathcal{C} := \bigcup_{\mathcal{T}_\omega \in \mathcal{C}_h} \Fh(\mathcal{T}_\omega)$. All elements that are not cut by $\partial \Omega$ or directly adjacent to cut elements form trivial patches $\mathcal{T}_\omega = \{T\}$. 
We will make the following assumption on $\mathcal{C}_h$.
\begin{assumption}\label{ass:patches}
  The set of disjoint patches $\mathcal{C}_h =
  \{\mathcal{T}_\omega\}$ ensures that for each element $T\in \ThG$ there is a
  patch $\mathcal{T}_\omega \in \mathcal{C}_h$ so that $T \in \mathcal{T}_\omega$.
  Further, the number of elements in a patch is uniformly bounded, s.t. from every element in a patch $\mathcal{T}_\omega$ the number of
  facets $F \in \mathcal{F}_h^\omega$ that needs to be crossed to approach
  the root element $T_\omega$ is also uniformly bounded (uniformly in $\mathcal{T}_\omega$, the
  cut position and $h$). 
\end{assumption}

For the construction of according patches we refer the interested reader to the
works of Badia and Verdugo on the \emph{aggregated} finite element method, see
e.g. \cite{BVM18}. For a sketch of the resulting mesh and patch configurations
see also \cref{fig:mesh}.

\begin{figure}
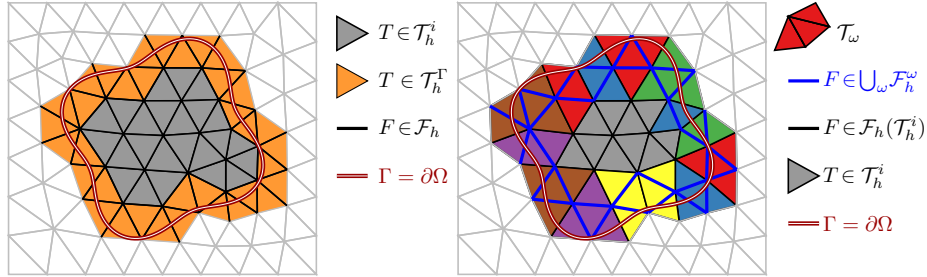

  \begin{center}\hspace*{-0.8cm}
\begin{tikzpicture}[
  dgtrig/.style={fill=Set1-I, semithick, draw=black},
  cluster0/.style={fill=Set1-A, semithick, draw=black},
  cluster1/.style={fill=Set1-B, semithick, draw=black},
  cluster2/.style={fill=Set1-C, semithick, draw=black},
  cluster3/.style={fill=Set1-D, semithick, draw=black},
  cluster4/.style={fill=Set1-G, semithick, draw=black},
  cluster5/.style={fill=Set1-F, semithick, draw=black},
  backgrtrig/.style={draw=gray!50!white, semithick},
  Gamma_h/.style={draw=none},
  dgtrig/.style={fill=Set1-I, semithick, draw=black},
  cuttrig/.style={fill=orange!80, thin, draw=gray!50!white},
  backgrtrig/.style={draw=gray!50!white, semithick},
  gpfacet/.style={draw=black, semithick},
  scale=0.9
  ]
  \begin{scope}
    \input{tikz_mesh}
    \draw[red!60!black, ultra thick, domain=0:360,smooth,variable=\phi,samples=100]
    plot ({ 2.25 + (1.35+0.2*sin(-4*\phi))*cos(\phi)},{2+(1.35+0.2*sin(-4*\phi))*sin(\phi)});
    \draw[white, ultra thin, domain=0:360,smooth,variable=\phi,samples=100]
    plot ({ 2.25 + (1.35+0.2*sin(-4*\phi))*cos(\phi)},{2+(1.35+0.2*sin(-4*\phi))*sin(\phi)});
    \end{scope}
    \begin{scope}[xshift=187,gpfacet/.style={draw=blue, very thick},
      ]
    \input{tikz_mesh_agg}
    \draw[red!60!black, ultra thick, domain=0:360,smooth,variable=\phi,samples=100]
    plot ({ 2.25 + (1.35+0.2*sin(-4*\phi))*cos(\phi)},{2+(1.35+0.2*sin(-4*\phi))*sin(\phi)});
    \draw[white, ultra thin, domain=0:360,smooth,variable=\phi,samples=100]
    plot ({ 2.25 + (1.35+0.2*sin(-4*\phi))*cos(\phi)},{2+(1.35+0.2*sin(-4*\phi))*sin(\phi)});
  \end{scope}

  \draw[dgtrig] (4.8, 3.3) -- (4.8, 3.8) -- (5.25, 3.55) --cycle;
  \node[right, scale=0.8] at (5.3, 3.55) {\color{black} $T\!\in\!\Thi$};
  \draw[cuttrig, draw=black] (4.8, 2.6) -- (4.8, 3.1) -- (5.25, 2.85) --cycle;
  \node[right, scale=0.8] at (5.3, 2.85) {\color{black} $T\in \ThG$};
  \draw[draw=black, very thick] (4.8, 2.15) -- (5.25, 2.15);
  \node[right, scale=0.8] at (5.3, 2.15) {\color{black} $F\!\in\!\Fh$};
  \draw[red!60!black, ultra thick] (4.8, 1.45) -- (5.25, 1.45);
  \draw[white, ultra thin] (4.8, 1.45) -- (5.25, 1.45);
  \node[right, scale=0.8] at (5.3, 1.45) {\color{red!60!black} $\Gamma = \partial \Omega$};
      
  \draw[cluster0] (11.2, 3.3) -- (11.4, 3.8) -- (11.65, 3.35) --cycle;
  \draw[cluster0] (11.4, 3.8) -- (11.65, 3.35) -- (12.0, 3.65) --cycle;
  \draw[gpfacet] (11.4, 3.8) -- (11.65, 3.35);
  \draw[cluster0] (11.4, 3.8) -- (12.0, 3.65) -- (11.8, 4) --cycle;
  \draw[gpfacet] (11.4, 3.8) -- (12.0, 3.65);
  \node[right, scale=0.8] at (12.0, 3.55) {\color{black} $\mathcal{T}_\omega$};

  \draw[dgtrig] (11.4, 1.2) -- (11.4, 1.7) -- (11.85, 1.45) --cycle;
  \node[right, scale=0.8] at (11.8, 1.45) {\color{black} $T\!\in \Thi $};
  \draw[draw=black, very thick] (11.4, 2.15) -- (11.85, 2.15);
  \node[right, scale=0.8] at (11.8, 2.15) {\color{black} $F\!\in\!\Fh(\Thi)$};
  \draw[blue, very thick] (11.4, 2.85) -- (11.85, 2.85);
  \node[right, scale=0.8] at (11.8, 2.85) {\color{blue} $F\!\in\!\bigcup_\omega \!\mathcal{F}_h^\omega$};
  \draw[red!60!black, ultra thick] (11.4, 0.75) -- (11.85, 0.75);
  \draw[white, ultra thin] (11.4, 0.75) -- (11.85, 0.75);
  \node[right, scale=0.8] at (11.8, 0.75) {\color{red!60!black} $\Gamma = \partial \Omega$};
\end{tikzpicture}\vspace*{-4cm}
\end{center}
\caption{Example of an unfitted geometry, the active and cut part of the
  computational mesh (left) and a set of patches as in \cref{ass:patches}.}\label{fig:mesh}
\end{figure}

Although practically relevant, in the presentation of the numerical methods and their analysis we will neglect the problem of geometry handling and approximation and assume that all geometrical operations, such as integrals on the unfitted domain $\Omega$, can be performed exactly. In \cref{sec:geom} we discuss how we deal with this problem in the numerical realization.

\subsection{Notation for finite element spaces}
On a set of elements $\mathcal{S}_h \subseteq \mathcal{T}_h$ with associated
domain $S$ we denote the space
of piecewise polynomials up to degree $k$ on that submesh as
\begin{align}
\mathbb{P}^k(\mathcal{S}_h) &:= \{ v \in L^2(S) \mid v|_T \in
  \mathcal{P}^k(T),~T\in\mathcal{S}_h\},
  \intertext{
where $\mathcal{P}^k(T)$ is the space of scalar
polynomials up to degree $k$ on $T$.
Similarly we denote the space of Raviart-Thomas functions up to degree $k$ by 
}
  \mathbb{R}\mathbb{T}^k(\mathcal{S}_h) &:= \{ v_h \in H(\div;S) \mid v_h|_T \in \mathcal{R}\mathcal{T}^k(T), T \in \mathcal{S}_h \},
\end{align}
where $\mathcal{R}\mathcal{T}^k(T) = [\mathcal{P}^k(T)]^d + \mathcal{P}^k(T)
\cdot \mathbf{x}$ is the local Raviart-Thomas space so that $\div
\mathbb{R}\mathbb{T}^k(\mathcal{S}_h) = \mathbb{P}^k(\mathcal{S}_h)$. Note that
in $ \mathbb{R}\mathbb{T}^k(\mathcal{S}_h)$ normal-continuity holds. Breaking up
normal-continuity leads to the broken Raviart-Thomas space that we denote by
$\mathbb{R}\mathbb{T}_-^k(\mathcal{S}_h)$. 
We further define by $\Pi^{\mathbb{R}\mathbb{T},k}: [H^1(S)]^d \to \mathbb{R}\mathbb{T}^k(\mathcal{S}_h)$ and $\Pi^{Q,k}: L^2(S) \to \mathbb{P}^k(\mathcal{S}_h)$ the usual Raviart-Thomas and $L^2$ interpolation operators and drop the index $k$ if the polynomial degree is clear from the context.
The introduced spaces and operators are well-known, well-understood and can be found in standard literature, e.g. \cite{brezzi2012mixed}. 

\subsection{Notation for inner products, norms, inequalities up to constants}

By $(\cdot,\cdot)_S$ we denote the $L^2(S)$ inner product on a domain $S$ and by $\Vert \cdot \Vert_S$ the corresponding $L^2(S)$ norm.
We use the notation $a \lesssim b$ if there exists a constant $c>0$, independent of the mesh size and mesh-interface cut position, such that $a\leq cb$. Similarly, we use $\gtrsim$ if $a\geq cb$, and $a\simeq b$ if both $a\lesssim b$ and $b \lesssim a$ holds.

\subsection{Model problem}
\noindent We consider Poisson's equation on $\Omega$ with Dirichlet boundary conditions on $\Gamma = \partial \Omega$: For a given functions $f: \Omega \to \mathbb{R}$, $p_D: \Gamma \to \mathbb{R}$, find $p: \Omega \to \mathbb{R}$ so that 
\begin{subequations}
    \begin{align}
        -\div(\nabla p) = \hphantom{-} f \text{ in } \Omega, &\qquad
        p = p_D \text{ on } \Gamma.
        \intertext{
The problem can be rewritten as a system of first order equations: Find ${u: \Omega \to  \mathbb{R}^d}$, $p: \Omega \to \mathbb{R}$ so that 
}
        u - \nabla p = 0 \text{ in } \Omega, \qquad
        \div u = - f \text{ in } \Omega, &\qquad p = p_D \text{ on } \Gamma.
        \label{eq:mixedPoisson}
    \end{align}
  \end{subequations}
  For the numerical treatment in the following we will focus on \eqref{eq:mixedPoisson} as it allows to put an emphasis on the \emph{conservation} of the (approximation of the) flux $u$.
  
We will focus on Dirichlet boundary conditions for simplicity in the main part of this manuscript and discuss the Neumann case in \cref{sec:neumann}.

\subsection{Weak formulation}
\noindent
As a basis for a finite element discretization we consider a variational formulation of \eqref{eq:mixedPoisson}. 
With $\Sigma := H(\div,\Omega)$ and $Q := L^2(\Omega)$,  
the weak formulation of \eqref{eq:mixedPoisson} reads:
For $f \in L^2(\Omega)$, $p_D \in H^{\frac12}(\Gamma)$ find $(u,p) \in \Sigma \times Q$ s.t.
\begin{align}
  a(u,v) + b(v,p) &= \langle v\cdot n, p_D \rangle && \forall v \in \Sigma, \tag{C-a}\label{eq:weakForm:a} \\
  b(u,q) &= -(f,q)_\Omega && \forall q \in Q. \tag{C-b}\label{eq:weakForm:b}
\end{align}
where $\langle \cdot , \cdot \rangle$ is the duality pairing between $H^{-\frac12}(\Gamma)$ and $H^{\frac12}(\Gamma)$. The bilinear forms $a(\cdot,\cdot)$ and $b(\cdot,\cdot)$ are defined as
\begin{align}
  a(u,v) &= (u,v)_\Omega, \qquad 
  b(u,q) = (\div u,q)_\Omega, \quad u,v \in \Sigma, q \in Q.
\end{align}
We identify the bilinear forms $a(\cdot,\cdot)$ and $b(\cdot,\cdot)$ with the corresponding operators $a: \Sigma \to \Sigma^\ast$ and $b: \Sigma \to Q^\ast$.
\subsection{Subproblems arising from a space decomposition}
Defining the divergence-free subspace of $\Sigma$ as
\begin{equation}
    \Sz := \ker b = \{v \in \Sigma \mid b(v,q) = 0 ~\forall q \in Q\} = \{ v \in \Sigma \mid \div v = 0 \text{ in } \Omega\}
  \end{equation}
  we obtain the orthogonal space decomposition $\Sigma = \Sz \oplus_a \Sperp$ where   
  $\Sperp$ is the $L^2(\Omega)$-orthogonal (i.e. $a$-orthogonal) complement of $\Sz$.
  We can now reformulate the continuous problem as the task of finding the three components $u^0 \in \Sz$, $u^\perp \in \Sperp$ and $p \in Q$ solving three subproblems, s.t. $(u^0 + u^\perp,p)$ solves \eqref{eq:weakForm:a}--\eqref{eq:weakForm:b}: 
\begin{align}
  &\text{Find } u^0\, \in \Sz\,,  \text{s.t. } &a(u^0,v^0) &= (v^0 \cdot n, p_D)_{\Gamma} && \forall v^0 \in \Sz, \tag{C-1} \label{eq:weakForm:I}\\[-0.5ex]
  &\text{Find } u^\perp \in \Sperp,  \text{s.t. } &b(u^\perp,q) &= -(f,q)_\Omega && \forall q \in Q, \tag{C-2} \label{eq:weakForm:II}\\[-0.5ex]
  &\text{Find } p^{\hphantom{\perp}} \in Q^{\hphantom{\perp}}, \text{s.t. } &b(v^\perp,p) &= (v^\perp \cdot n,p_D)_{\Gamma} - a(u^\perp,v^\perp) && \forall v^\perp \in \Sperp. \tag{C-3} \label{eq:weakForm:III}
\end{align}
We will exploit this characterization of the mixed formulation in the remaining considerations of this work to investigate discrete variational formulations, especially with respect to their consistency. 

\section{Unfitted discretizations} \label{sec:discretization}
In this section we want to derive new unfitted discretizations for \eqref{eq:weakForm:a}--\eqref{eq:weakForm:b} with a focus on preserving conservation properties. 
In \cref{sec:restrictedmixedfem} we start with a straight-forward unfitted mixed FEM that is obtained by directly posing \eqref{eq:weakForm:a}--\eqref{eq:weakForm:b} on finite dimensional subspaces of $\Sigma$ and $Q$ and work out the need for corrections.
We will then adapt the discrete variational formulation by replacing $a(\cdot,\cdot)$, $b(\cdot,\cdot)$ and $f$ by discrete versions $a_h(\cdot,\cdot)$, $b_h(\cdot,\cdot)$ and $f_h$. We will introduce these modifications successively in Sections \ref{sec:newbh}--\ref{sec:gp} yielding the new unfitted mixed formulation presented in \cref{sec:unfmixed}. In \cref{sec:consistency} we discuss the impact of the modifications introduced in terms of consistency and we relate the method to a similar approach in the literature in \cref{sec:relation:to:others}.

\subsection{Restricted mixed FEM} \label{sec:restrictedmixedfem}
We set
$
\Sh = \mathbb{R}\mathbb{T}^k(\mathcal{T}_h) \subset H(\div,\OmT)
$
and
$
Q_h = \mathbb{P}^k(\Th) \subset L^2(\OmT)
$
\noindent and start with a straight-forward application of mixed finite element methods in an unfitted setting. We take the finite element spaces w.r.t. the active mesh $\Th$, but restrict the integrals to $\Omega$. Hence, we denote this method as \emph{restricted} mixed FEM.
The discretization reads as: Find $(u_h,p_h) \in \Sh \times Q_h$ such that for all $(v_h,q_h) \in \Sh \times Q_h$ there holds
\begin{align} \label{eq:mixedprob} \tag{R}
    a(u_h,v_h) + b(v_h,p_h) = (v_h \cdot n,p_D)_{\Gamma}, \qquad %
    b(u_h,q_h) = -(f,q_h)_\Omega. %
\end{align}

\noindent 
According to the orthogonal decomposition 
$
    \Sh = \Shz \oplus_a \Shperp
$
with the discrete subspaces
$
    \Shz := \ker b 
    = \{ v_h \in \Sh \mid \div v_h = 0\}$ and $\Shperp := \{v_h \in \Sh \mid a(v_h,w_h) = 0 ~ \forall w_h \in \Shz\}$,
we can reinterpret problem \eqref{eq:mixedprob} as the three subproblems: 
\begin{align}
  &\text{Find } u_h^{0\,} \in \Shz, \text{s.t. } &a(u_h^0,v_h^0) &= (v_h^0 \cdot n, p_D)_{\Gamma} && \forall v_h^0 \in \Shz, \tag{R-1} \label{eq:mixedprob:I}\\[-0.5ex]
  &\text{Find } u_h^\perp \in \Shperp, \text{s.t. } &b(u_h^\perp,q_h) &= -(f,q_h)_\Omega && \forall q_h \in Q_h, \tag{R-2} \label{eq:mixedprob:II}\\[-0.5ex]
  &\text{Find } \,p_{h\,} \in\! Q_h, \text{s.t. } &b(v_h^\perp,p_h) &= (v_h^\perp \cdot n,p_D)_{\Gamma} - a(u_h^\perp,v_h^\perp) && \forall v_h^\perp \in \Shperp. \tag{R-3} \label{eq:mixedprob:III}
\end{align}
Let us take a look at these three subproblems to characterize the stability of \eqref{eq:mixedprob}.
Stability of subproblem \eqref{eq:mixedprob:I} follows directly from coercivity of $a(\cdot,\cdot)$ on $\Shz$ w.r.t. $\Vert \cdot \Vert_{H(\div;\Omega)}$.
With $v_h^0 = u_h^0$ in \eqref{eq:mixedprob:I} we have $\Vert u_h^0 \Vert_{H(\div;\Omega)} \leq \Vert p_D \Vert_{H^{1\!/\!2}(\Gamma)}$ and with $q_h = \div u_h^\perp \in Q_h$ in \eqref{eq:mixedprob:II} we obtain $\Vert \div u_h^\perp \Vert_{\Omega} \leq \Vert f \Vert_{\Omega}$.
This yields control of $u_h \in \Sh$ by the r.h.s. data only in the (unsatisfactorily) weak norm $\Vert u_h^0 \Vert_\Omega + \Vert \div u_h \Vert_\Omega$. To obtain control (and hence finally error bounds) in $\Vert \cdot \Vert_\Omega$ (or $\Vert \cdot \Vert_{H(\div;\Omega)}$) we are missing an estimate $\Vert \div v_h^\perp \Vert_{\Omega} \gtrsim \Vert v_h^\perp \Vert_{\Omega}$ for all $v_h \in \Shperp$. %
More obviously problematic is the stability of \eqref{eq:mixedprob:III} which requires a stability result of the form 
\begin{equation} \label{eq:rmfem:infsup}
  \inf_{q_h \in Q_h \setminus \{0\}} \sup_{v_h \in \Shperp\setminus \{0\}} \frac{b(v_h,q_h)}{\Vert v_h \Vert_{H(\div;\Omega)} \Vert q_h \Vert_{\Omega}} > c_{R}
\end{equation}
with a constant $c_R > 0$. One can indeed easily show that \eqref{eq:rmfem:infsup} holds, cf. \cref{lem:rmfem:infsup} in \cref{sec:rmfem:infsup}. With standard saddle point theory, cf. \cite{BVM18}, this implies unique solvability of \eqref{eq:mixedprob} and stability in $\Vert u_h \Vert_{H(\div;\Omega)} + \Vert p_h \Vert_{\Omega}$ for a constant $c_R$. However, the constant $c_R$ will \emph{not be robust w.r.t. the cut position} and may become arbitrary small, which renders the method practically useless.

In the next sections we derive a modified scheme tailored to fix the problem with the inf-sup constant.

\subsection{The discrete bilinear form $b_h(\cdot,\cdot)$} \label{sec:newbh}
Motivated by the insufficient robustness of $b(\cdot,\cdot)$ w.r.t. the inf-sup condition, we modify the bilinear form $b(\cdot,\cdot)$ by substituting the previous domain of integration $\Omega$ with the domain of the active mesh $\OmT$:
\begin{equation}
    b_h(v_h,p_h) := (\div v_h,p_h)_{\OmT}. 
\end{equation}
The same type of modification for the setting of a low-order Stokes discretization has been considered in \cite{burman2022cut}.
By construction, this form does not depend on the local cut configuration and will allow to provide control on the active mesh. Further, there holds $\Shz = \ker b = \ker b_h$ so that the reduced problem on $\Shz$ like \eqref{eq:mixedprob:I} will not be affected by this change.
Subproblems \eqref{eq:mixedprob:II} and \eqref{eq:mixedprob:III} will obviously
be affected severely. For \eqref{eq:mixedprob:II} we will change the r.h.s. to
adapt for the change in the integration domain below in \cref{sec:discretefh}.
For \eqref{eq:mixedprob:III} we will accept the fact that the resulting discrete
approximation of the primal unknown $p$ will be inconsistent. In
\cref{sec:consistency} we will make sense of the inconsistent discrete
approximation of $p$ and will afterwards re-obtain a consistent approximation from post-processings discussed in \cref{sec:post-process}. 

\subsection{The extended r.h.s. data $f_h$} \label{sec:discretefh}
In the discretization below we will replace $b(\cdot,\cdot)$ by $b_h(\cdot,\cdot)$ which changes the integration domain for the discrete conservation equation. Correspondingly, we need to adjust the r.h.s. integral and require a proper integrand $f_h: \OmT \to \mathbb{R}$, which we assume to be known.
This is reasonable, as we can typically consider to be in one of two situations:
\begin{itemize}
  \item 
    In the first case the function $f$ is given as a finite element function on the background mesh, because it stems from a coupling to a different (unfitted) discrete field or it is prescribed in closed form anyway.
  \item In the second case $f: \Omega \to \mathbb{R}$ is not given with a proper extension to $\OmT$ and the discrete extension has to be constructed explicitly, cf. \cref{sec:gp}
\end{itemize}    
Why an extension of $f$ to $\OmT$ is necessary is further explained in \cref{rem:need:ext}.

\subsection{Ghost penalty stabilization} \label{sec:gp}
For the restricted mixed FEM we have seen that the bilinear form $a(\cdot,\cdot)$ suffices for stability w.r.t. $L^2(\Omega)$, respectively $H(\div;\Omega)$. 
However, there are good reasons to ask for control on the whole active mesh, i.e. in $L^2(\OmT)$, respectively $H(\div;\OmT)$. For instance, this can be crucial for conditioning or finite element post-processing. This control can be obtained by adding a proper stabilization term to the bilinear form $a_h(\cdot,\cdot)$, the \emph{ghost penalty} (GP) stabilization, cf. \cite{Bur10}.
Another application of the GP stabilization in the context of this work is to obtain a proper extension of the r.h.s. data $f_h$ to the whole active mesh $\OmT$ in the case where $f$ is only given on $\Omega$, cf. \cref{sec:discretefh}.

We denote a GP stabilization, which is a symmetric non-negative bilinear form for a generic piecewise polynomial finite element space $R_h \subseteq \mathbb{P}^k(\Th)$ (respectively $R_h \subseteq [\mathbb{P}^k(\Th)]^d$) by $j_h(\cdot,\cdot)$ with corresponding semi-norm $|\cdot|_{j}$.
Although different GP formulations exist, we restrict to the set of methods that split into facet contributions, so that
$j_h(u_h,v_h) = \sum_{F \in \mathcal{F}_h^\ast} j_F(u_h,v_h)$, where
$\mathcal{F}_h^\ast$ is a set of facets and $j_F(\cdot,\cdot)$ is a
stabilization term acting on facet $F$ or the facet patch $\omega_F$. 
We make the following assumptions, similar to e.g. \cite[Assumptions E2,E3]{gurkan2020}:
\begin{assumption}\label{ass:gp}
    Let $T_1, T_2$ be the two adjacent elements to a facet $F \in \mathcal{F}_h^\ast$ and 
    $r_h \in \mathbb{P}^k(\{T_1,T_2\})$ be an element-wise polynomial function up to degree $k$.     
    There holds the local stability property
    \begin{equation} \label{eq:gp0} \tag{GP1a}
      j_F(r_h,r_h) \gtrsim \inf_{v_h \in \mathcal{P}^{k}(\omega_F)} \Vert v_h - r_h \Vert^2_{\omega_F}
    \end{equation}
    with (hidden) constants independent of $h$ or $F$, i.e. the GP term is lower bounded by the distance to a patchwise polynomial of the same degree.  
    We further assume that $j_F(\cdot,\cdot)$ is continuous w.r.t. the $L^2$-norm so that a consequence of \eqref{eq:gp0} is that for $r_h \in R_h$ there holds the global stability property
    \begin{equation} \label{eq:gp} \tag{GP1b}
        \Vert r_h \Vert_{\Omega} + \vert r_h \vert_{j} \simeq \Vert r_h \Vert_{\OmT}.
      \end{equation}
      Further for $r \in H^m(\OmT) \subset \mathcal{V}$ for some Sobolev space $\mathcal{V}$ and $\mathcal{I}_R : \mathcal{V} \to R_h$, $m \in \{0,\dots,k+1 \}$ a suitable interpolation operator into $R_h$ there holds the weak consistency property 
      \begin{equation} \label{eq:gp2} \tag{GP2}
        \vert \mathcal{I}_R r \vert_{j} 
        \lesssim 
        h^{m} \vert r \vert_{H^m(\Omega^{\mathcal{T}}_j)}\leq 
        h^{m} \vert r \vert_{H^m(\OmT)}
    \end{equation} 
    where $\Omega^{\mathcal{T}}_j \subset \OmT$ is the subdomain of elements on which the GP stabilization acts.
\end{assumption}

In the remainder we set $\mathcal{F}_h^\ast =
\mathcal{F}_h^{\mathcal{C}}$ which limits the amount of additional couplings that is introduced by the GP mechanism, cf. also \cite{BNV22}.
Two popular candidates for $j_F(\cdot,\cdot)$ that fulfill \cref{ass:gp} are given in \cref{app:gp}.
Let us note that the GP stabilization bilinear form may depend on the polynomial degree of the finite element space $R_h$. Furthermore, with $\mathcal{F}_h^\ast =
\mathcal{F}_h^{\mathcal{C}}$ the domain $\Omega^{\mathcal{T}}_j$ in \eqref{eq:gp2} decomposes into the disjoint nontrivial patches of $\mathcal{C}_h$, and we can deduce that a function $r$ that is a polynomial up to degree $k$ on each of these patches is in the kernel of $j_h(\cdot,\cdot)$ as $r = \mathcal{I}_R r$ and $|r|_{H^{k+1}(\Omega_j^{\mathcal{T}})} = 0$.    

We will make use of the GP stabilization in the following in up to three occasions:
First, to potentially enlarge the domain of control. For $\gamma_u \ge 0$, we define
\begin{equation} \label{def:ah}
    a_h(u_h,v_h) := a(u_h,v_h) + \gamma_u j_h(u_h,v_h), \quad u_h,v_h \in \Sh.
\end{equation}
Note that %
we also allow $\gamma_u = 0$, i.e. the unstabilized form where $a_h(\cdot,\cdot) = a(\cdot,\cdot)$. 

Second, for the generic construction of $f_h$ -- if needed -- we propose the following discrete variational problem: To $\gamma_f > 0$ find $f_h \in Q_h^{k_f} = \mathbb{P}^{k_f}(\Th)$ with $k_f \in \mathbb{N}_0$, s.t.
  \begin{align} \label{eq:fh}
    (f_h, q_h)_{\Omega} + \gamma_f j_h(f_h, q_h) = (f,q_h)_{\Omega}, \quad \forall ~ q_h \in Q_h^{k_f}.
  \end{align}
The obvious default choice for $k_f$ is $k$ so that $\div \Sh = Q_h = Q_h^k = Q_h^{k_f}$, we will, however, also allow $k_f \neq k$ in the following.
Note that the matrix corresponding to the l.h.s. is block-diagonal, with each block corresponding to the unknowns of one patch. Note that this is a specific consequence of the choice $\mathcal{F}_h^\ast = \mathcal{F}_h^{\mathcal{C}}$.
The solution of \eqref{eq:fh} can then be obtained by solving a sequence of local problems (in parallel).

Finally, in \cref{sec:patchpost-process} we will make use of the GP stabilization for one of our post-processing schemes again.

\subsection{The unfitted mixed FEM} \label{sec:unfmixed}
\noindent Bringing the previous motivation into play, we formulate the following unfitted mixed FEM problem: Find $(u_h,\bar p_h) \in \Sh \times Q_h$ such that 
  \hypertarget{def:M}{
    \begin{align}
         a_h(u_h,v_h) + b_h(v_h,\bar p_h) &= (v_h \cdot n,p_D)_{\Gamma} && \forall v_h \in \Sh, \label{eq:new:mixedprob:a} \tag{M-a}\\
         b_h(u_h,q_h) &= -(f_h, q_h)_{\OmT} && \forall q_h \in Q_h.  \label{eq:new:mixedprob:b} \tag{M-b} \\
        \intertext{\noindent 
        As before, we split the problem into three subproblems.
        However, with $a(\cdot,\cdot)$ being replaced by $a_h(\cdot,\cdot)$ we also change the definition of the discrete space $\Shperp$ to
        $
        \Shperp = \{ v_h \in \Sh \mid a_h(v_h,w_h) = 0 \quad \forall w_h \in \Shz \}.
        $        
        }
    \text{Find } u_h^0 \, \in \Shz \text{, s.t. } a_h(u_h^0,v_h^0) &= (v_h^0 \cdot n, p_D)_{\Gamma} && \forall v_h^0 \in \Shz, \tag{M-1} \label{eq:new:mixedprob:I}\\[-0.5ex]
    \text{Find } u_h^\perp \in \Shperp\text{, s.t. } b_h(u_h^\perp,q_h) &= -(f_h,q_h)_{\OmT} && \forall q_h \in Q_h, \tag{M-2} \label{eq:new:mixedprob:II}\\[-0.5ex]
    \text{Find } \bar p_h^{\hphantom{\perp}} \in Q_h \text{, s.t. } b_h(v_h^\perp, \bar p_h) &= (v_h^\perp \cdot n,p_D)_{\Gamma} - a_h(u_h^\perp,v_h^\perp) && \forall v_h^\perp \in \Shperp. \tag{M-3} \label{eq:new:mixedprob:III}
\end{align}
  }
The advantage of this discretization over the restricted mixed FEM \eqref{eq:mixedprob} comes with $b_h(\cdot,\cdot)$ providing proper control on the whole active mesh. This comes at the price of an inconsistency that is discussed in the next section. 
\begin{remark} \label{rem:need:ext}
  From \eqref{eq:new:mixedprob:b} it becomes clear that we need a proper extension of the r.h.s. data $f$ to $\OmT$ to obtain a consistent discretization. If $f$ is only extended trivially, i.e. by zero, outside of $\Omega$ the discrete solution would face the approximation problem $\div u_h = \chi_{\Omega} \cdot f$ with a discontinuous r.h.s. (unless $f|_{\partial \Omega}\equiv 0$) on $\OmT$ which would result in deteriorated convergence order of expected order $\mathcal{O}(h^{\frac12})$. 
\end{remark} 

\subsection{Consistency of the unfitted mixed FEM} \label{sec:consistency}
We will analyze potential inconsistencies resulting from our modifications by sequentially reviewing problems \eqref{eq:new:mixedprob:I}, \eqref{eq:new:mixedprob:II}, and \eqref{eq:new:mixedprob:III}.
While \eqref{eq:new:mixedprob:I} is affected -- if at all -- by the weakly consistent GP term, in \eqref{eq:new:mixedprob:II}, we avoid inconsistencies -- up to a possible approximation error $f_h \approx f$ -- in the mass balance by replacing $\div v_h = f$ on $\Omega$ by $\div v_h = f_h$ on the larger domain $\OmT$.
The major change in consistency comes with
\eqref{eq:new:mixedprob:III}. 
Indeed, $\bar p_h$ is not a good approximation of
$p$ on cut elements. However, as $\bar{p}_h$ as the solution of \eqref{eq:new:mixedprob:III} does not influence \eqref{eq:new:mixedprob:I} or \eqref{eq:new:mixedprob:II} this inconsistency does not affect $u_h \in \Sh$. 
 We can also make some sense of the solution $\bar{p}_h$ of
 \eqref{eq:new:mixedprob:III} as a reasonable approximation of a slightly
 different field. For a discrete function in $L^2(\Omega)$ let us define the
 $L^2(\OmT)$ projection into $Q_h$ of the extension by zero by $\mathcal{E}_h^0: L^2(\Omega) \to Q_h, q \mapsto \Pi^Q (\chi_{\Omega} \cdot q)$, so that
\begin{equation} \label{eq:extbyzeroh}
  ( \mathcal{E}_h^0 q, r_h)_{\OmT} = (\chi_{\Omega} \cdot q,r_h)_{\OmT} = (q,r_h)_{\Omega} \quad \forall r_h \in Q_h.
\end{equation}
This allows us to write $b(v_h,q) = b_h(v_h,\mathcal{E}_h^0 q)$ and interpret
$\bar p_h$ as an approximation of the extension by zero of $p$, $\bar p_h
\approx \mathcal{E}_h^0 p$. A similar interpretation of $\bar p_h$ has been given in \cite{burman2022cut} for their unfitted discretization of the Stokes problem.
Indeed, the following consistency relation holds for all $w_h \in \Sh$:
\begin{align}
    a_h(u_h,w_h) & -a(u,w_h)  = a(u_h-u,w_h) + \gamma_u j_h(u_h,w_h) \nonumber \\ &= b(w_h,p) - b_h(w_h,\bar p_h) 
    = b_h(w_h,\mathcal{E}_h^0 p - \bar p_h) = (\div w_h, \mathcal{E}_h^0 p - \bar p_h)_{\OmT} \label{eq:consistency}
\end{align}
This characterization will also be crucial in the numerical analysis below.

\begin{remark} 
    We note that $\mathcal{E}_h^0$ is invertible (as long as $\meas_d(T\cap\Omega) > 0$ for all $T \in \ThG$), but the norm of the inverse of $\mathcal{E}_h^0$ strongly depends on the cut configuration, making it necessary to avoid including $(\mathcal{E}_h^0|_{Q_h})^{-1}$ in a numerical realization. 
\end{remark}

In mixed finite element methods, post-processing is often used to obtain an additional order of accuracy. We will use post-processing here as well, but to achieve two goals at the same time: First, to repair the inconsistency of $\bar{p}_h$ on cut elements, and second, to obtain an additional order of accuracy for uncut and cut elements. Two possible post-processing schemes are introduced in \cref{sec:post-process} after the analysis of the main method in \cref{sec:analysis1}. 

\subsection{Relation to the method in Frachon et al. (\cite{frachon2022divergence})} \label{sec:relation:to:others}
In \cite{frachon2022divergence} unfitted mixed FE methods for the Darcy interface problem are introduced that are based on a stabilization of the bilinear form $b(\cdot,\cdot)$ based on the GP mechanisms. The suggested scheme, translated to the one-domain case, leads to the following discrete problem:
Find $(u_h,p_h) \in \Sh \times Q_h$ such that
\begin{subequations}\label{eq:frachmethod}
\begin{align}
    a_h(u_h,v_h) + b_h^\ast(v_h, p_h) &= (v_h \cdot n,p_D)_{\Gamma} && \forall v_h \in \Sh,  \\ %
    b_h^\ast(u_h,q_h) &= -(f, q_h)_{\Omega} && \forall q_h \in Q_h, \label{eq:frachmethodb} %
\end{align}        
\end{subequations}
where $b_h^\ast(u_h,q_h) = (\div u_h,q_h)_\Omega + \gamma_{\div} j_h(\div u_h,q_h)$.
This scheme is closely related to the method \eqref{eq:new:mixedprob:a}--\eqref{eq:new:mixedprob:b} with $f_h$ from \eqref{eq:fh} and $\gamma_f = \gamma_{\div}$, as we will show in the following lemma.

\begin{lemma} \label{lem:relation:to:others}
    The method \eqref{eq:new:mixedprob:a}--\eqref{eq:new:mixedprob:b} with $f_h$ from \eqref{eq:fh} and $\gamma_f = \gamma_{\div}$ and the method \eqref{eq:frachmethod} from \cite{frachon2022divergence} yield the same solution $u_h \in \Sh$ and the same solution $p_h \in Q_h$ (and $\bar p_h \in Q_h$ respectively) away from cut elements and their direct neighbors.
\end{lemma}
\begin{proof}
  From the definition of \eqref{eq:fh} we obtain in \eqref{eq:frachmethodb} that 
  $$ (\div u_h,q_h)_\Omega + \gamma_{\div} j_h(\div u_h,q_h) = - (f_h,q_h)_\Omega - \gamma_{f} j_h(f_h,q_h) \quad \forall q_h \in Q_h.
$$
With $q_h = \div u_h \in Q_h$ and $\gamma_{\div} = \gamma_f$ we obtain that $\div u_h$ is uniquely determined and the unique solution for $\div u_h$ is $\div u_h = f_h$ in $\OmT$. The same obviously holds for \eqref{eq:new:mixedprob:b}. The space decomposition as well as problem \eqref{eq:new:mixedprob:I} are the same for both problems so that $u_h$ is the same in both problems. On elements away from cut element and their direct neighbors, the local problems to solve for $p_h$ and $\bar p_h$, respectively, are also the same yielding the same pressure approximations there. Note that on cut elements and direct neighbors in \cite{frachon2022divergence} there is a GP term active which is not acting on $\bar p_h$ so that in general $\bar p_h \neq p_h$ on these elements.
\end{proof}
\begin{remark}[Patch-local integral conservation property]
  Let us note that for both methods, as well as for the method in \cite{puppi2021cut} with a patch-localized version of the ghost penalty there holds the patch-local integral conservation property $\int_{\Omega \cap \omega} \div u_h \, dx = \int_{\Omega \cap \omega } f \, dx$ for all ghost penalty patches (or interior elements) $\omega$ as the ghost penalty terms vanish for the choice $q_h = \chi_{\omega}$ with $\chi_{\omega}$ the indicator function to the patch $\omega$.
\end{remark}

Note that the previous lemma together with the subsequent analysis of the method \eqref{eq:new:mixedprob:a}--\eqref{eq:new:mixedprob:b} in \cref{sec:analysis1} also implies a convergence analysis of the method \eqref{eq:frachmethod}. This way, the analysis presented here slightly extends the analysis in \cite{frachon2022divergence}, especially with a different viewpoint on the LBB-analysis and a duality result. Further, the analysis reveals that the post-processings proposed below can also be applied for the method \cite{frachon2022divergence}. Note that although the approaches are very similar, the numerical computation of \eqref{eq:frachmethod} comes at a slightly higher computational cost as the GP stabilization term in $b_h^\ast(\cdot,\cdot)$ introduces additional couplings and thus increases the number of non-zero entries in the system matrix, cf. \cref{sec:compstab}. In contrast the method \eqref{eq:new:mixedprob:a}--\eqref{eq:new:mixedprob:b} applies -- if necessary -- 
a sequence of patch-wise GP stabilized $L^2$ projections that are small cheap computations that can be executed in parallel. 
We numerically verified the equivalence of the solutions as stated in \cref{lem:relation:to:others} up to round-off errors. The interested reader is invited to make use of the reproduction data, cf. \hyperlink{sec:dataavail}{section on data availability}, which includes the implementation of both methods. 
 
\section{A-priori error analysis of the unfitted mixed method \eqref{eq:new:mixedprob:I}--\eqref{eq:new:mixedprob:III}} \label{sec:analysis1}
After some preliminaries on extensions, the approximation of the r.h.s. and the norms considered in the analysis in \cref{sec:analysis:preliminaries} and \cref{sec:analysis:norms}, we will discuss an a priori error analysis of the discretization \eqref{eq:new:mixedprob:a}--\eqref{eq:new:mixedprob:b} with a focus on $u_h$ in \cref{sec:analysis:apriori1}. In \cref{sec:analysis:post-proc} we then analyse the accuracy of post-processings for $p_h$.

\subsection{Preliminaries: Sobolev extensions and discrete extension $f_h$ of $f$} \label{sec:analysis:preliminaries}
First, to make sense of $u$, $p$ and $f$ on $\OmT$ we will assume the existence of proper extensions to $\OmT$.
We use the notation for a bounded linear extension operator $\mathcal{E}:\,
H^{m}(\Omega) \to H^{m}(\OmT),~m \in \mathbb{N}$, but note that the extension is
not necessarily unique.
A possible extension operator is discussed for instance in \cite[Theorem II.3.3]{GaldibookNS}.
We introduce $\pe = \mathcal{E} p$ and derive from this $\ue = \nabla \pe$ and
$\fe = - \div \ue$, i.e. the extensions of $u$, $f$ and $p$ are chosen consistently.
The need for an approximation of an extension of $f$ on $\OmT$ in an unfitted setting has a significant impact on the regularity requirements in the analysis compared to geometrically fitted mixed methods. While this affects only the regularity of $f$ at first glance, the regularity of $u$ and $p$ is affected through the PDE equations $- \div u = f$, $u = \nabla p$, especially there holds $\Vert f \Vert_{H^{\ell}(\Omega)} = \Vert \div u \Vert_{H^{\ell}(\Omega)} = \Vert \Delta p \Vert_{H^{\ell}(\Omega)} \lesssim \Vert p \Vert_{H^{\ell+2}(\Omega)}$ for $\ell \in \mathbb{Z}$.
We will discuss the difference to geometrically fitted methods in that regard in more detail in the next remark. 

\begin{remark} \label{rem:freg}
  The approximation of the extension of $f$ on $\OmT$ leads to a subtle but important difference compared to most geometrically fitted mixed methods. In geometrically fitted mixed methods, one typically has $ f_h = \Pi^Q f = - \Pi^Q \div u = - \div u_h$ on $\OmT$ so that $\div u - \div u_h$ is $L^2$-orthogonal to $Q_h$. Furthermore, the Raviart-Thomas interpolator $v_h = \PiRT u$ also has $\div v_h = \div \PiRT u = \Pi^Q \div u$ so that $\div u_h - \div v_h = 0$. This is in general not the case for solutions to \eqref{eq:new:mixedprob:a}--\eqref{eq:new:mixedprob:b} as $\div u_h = f_h \neq \Pi^Q \fe$. In the subsequent analysis, we will hence observe terms related to $f_h - \fe$ and $\div u_h - \div \ue$, respectively, showing up in several estimates. 
  To get proper bounds for $f_h - \fe$ we will require one additional order of regularity for $f$ compared to standard literature for geometrically fitted problems. Regularity of $f$ can be implied from regularity of $p$ (and $u$) and it would suffice to assume one additional order of regularity for $p$ and $u$ compared to standard literature to obtain optimal order results. However, we will distinguish the regularity for $f=\div u$ and $(u,p)$ for the most part of the analysis to keep the origin of the additional regularity requirement transparent.
\end{remark}

We formulate an explicit assumption on the approximation of $\fe$ by $f_h$:
\begin{assumption}\label{ass:fh}
    We assume that a discrete approximate extension of the source function $f$
    from $\Omega$ to $\OmT$ is given by $f_h: \OmT \to \mathbb{R}$ and that it
    is a good approximation to $\fe := -\div \ue$, so that
    there holds for $r \in  \{ 0, .., k_f+1\}$ with $k_f \in \mathbb{N}_0$
    \begin{equation}\label{eq:ass:fh} \tag{$E_f$}
        \Vert f_h - \fe \Vert_{\OmT} + h^{-1} \Vert f_h - f \Vert_{-2} \lesssim h^{r} \Vert f \Vert_{H^{r}(\Omega)}
    \end{equation}
    where the $-2$-norm denotes the operator norm for functionals over $H^2(\Omega)\cap H^1_0(\Omega)$.\footnote{This norm is weaker then the  $H^2(\Omega)'$-norm and stronger then the $H^{-2}(\Omega)$-norm. } 
\end{assumption}

\begin{lemma}\label{lem:fh:stabl2}
If \eqref{eq:fh} with $\gamma_f > 0$ is used to define $f_h \in Q_h^{k_f}$ \cref{ass:fh} is fulfilled.
\end{lemma}
\begin{proof}
    Let $g_h \in Q_h^{k_f}$ be arbitrary. Then, from the definition \eqref{eq:fh} we directly obtain
    \begin{align}
        \Vert f \!-\! f_h \Vert_\Omega^2 + \gamma_f \vert f_h \vert_j^2 & = (f\!-\!f_h,g_h\!-\!f_h)_{\Omega} + (f\!-\!f_h,f\!-\!g_h)_{\Omega} + \gamma_f |f_h|_j^2 \nonumber
        \\[-1ex]
        & \stackrel{\eqref{eq:fh}}{=} \gamma_f j_h(f_h,g_h) + (f\!-\!f_h,f\!-\!g_h)_{\Omega} \nonumber \\
\Longrightarrow \quad   \Vert f \!-\! f_h \Vert_\Omega + \gamma_f^{\frac12} \vert f_h \vert_j & \lesssim \Vert f \!-\! g_h \Vert_\Omega + \gamma_f^{\frac12} \vert g_h \vert_j \label{eq:stabl2proj:a}.
    \end{align}
    Exploiting \eqref{eq:gp} then yields a bound on $\OmT$
    \begin{align}
       \Vert \fe \!\! & -\! f_h \Vert_{\OmT} = \nrm{\fe \!\!-\! g_h}_{\OmT} + \nrm{g_h \!-\! f_h}_{\OmT}  \lesssim %
       \nrm{\fe \!\!-\! g_h}_{\OmT} + \nrm{g_h \!-\! f_h}_{\Omega} + |g_h\!-\!f_h|_j \nonumber \\[-1ex]
       & \lesssim \nrm{\fe \!\!-\! g_h}_{\OmT} + \nrm{g_h \!-\! f}_{\Omega} + |g_h|_j + \nrm{f_h \!-\! f}_{\Omega} + |f_h|_j \stackrel{\eqref{eq:stabl2proj:a}}{\lesssim} \nrm{\fe \!\!-\! g_h}_{\OmT} + |g_h|_j\nonumber \\
       &  \lesssim h^{\ell-1} \nrm{\fe}_{H^{\ell-1}(\OmT)} \lesssim h^{\ell-1} \nrm{f}_{H^{\ell-1}(\Omega)}, 
       \nonumber
    \end{align}
    $\ell=1,...,k_f+2$, where the last step follows with $g_h = \Pi^Q \fe$, the approximation properties of $Q_h$ and $\eqref{eq:gp2}$.
Finally, we prove the $-2$-norm-bound. By definition we have 
\begin{align}
    \Vert f \!\! & -\! f_h \Vert_{-2} = \sup_{v\in H^2(\Omega)\cap H^1_0(\Omega) \setminus \{0\} } \frac{(f-f_h,v)_\Omega}{\nrm{v}_{H^2(\Omega)}}.
    \nonumber
 \end{align}
 To $v \in H^2(\Omega)$ we denote by $v_h^\ast \in Q_h^{k_f}$ the solution of $(v_h^\ast,q_h)_\Omega + \gamma_f j_h(v_h^\ast,q_h) = (v,q_h)_\Omega$, i.e. $v_h^\ast$ is the stabilized $L^2$ projection of $v$ as in \eqref{eq:fh} with $\gamma_f > 0$ and $k_f$. Then, we have with \eqref{eq:fh} that
\begin{align*}
(f\!-\!f_h,v)_\Omega\! & = (f\!-\!f_h,v\!-\!v_h^\ast)_\Omega\! + (f\!-\!f_h,v_h^\ast)_\Omega\! = (f\!-\!f_h,v\!-\!v_h^\ast)_\Omega\! + \gamma_f j_h(f_h,v_h^\ast) \\
& \lesssim (\nrm{f\!-\!f_h}_\Omega\! + \gamma_f^{\frac12} |f_h|_j) (\nrm{v\!-\!v_h^\ast}_\Omega\! + \gamma_f^{\frac12} |v_h^\ast|_j)
\lesssim h^{r} \nrm{f}_{H^{r-1}(\Omega)} \Vert v \Vert_{H^2(\Omega)}
\end{align*}
where we exploited $\nrm{v-v_h^\ast}_\Omega + \gamma_u^{\frac12} |v_h^\ast|_j \lesssim h \Vert v \Vert_{H^1(\Omega)} \lesssim h \Vert v \Vert_{H^2(\Omega)}$ with arguments as in \eqref{eq:stabl2proj:a} and $k_f \geq 0$.
Hence, the result for the $-2$-norm also holds true which concludes the proof.
\end{proof}

\begin{remark}\label{rem:fass}
  In the previous proof one can also use the bound $\nrm{v-v_h^\ast}_\Omega + \gamma_u^{\frac12} |v_h^\ast|_j \lesssim h^2 \Vert v \Vert_{H^2(\Omega)}$
  for $k_f \geq 1$ and $r>1$ in \cref{ass:fh} yielding the stronger bound: 
  \begin{equation*}
    \Vert f_h - \fe \Vert_{\OmT} + h^{-2} \Vert f_h - f \Vert_{-2} \lesssim h^{r} \Vert f \Vert_{H^{r}(\Omega)},
\end{equation*}
 however, since the $\Vert f_h - \fe \Vert_{\OmT}$ term dominates in what follows, we only consider the weaker bound in Assumption \ref{ass:fh}, but don't have to treat the case $k_f = 0$ separately.
\end{remark}

\subsection{Norms for the error analysis} \label{sec:analysis:norms}
In the analysis we will avoid applying the GP bilinear form (or the GP semi-norm $|\cdot|_j$) to functions in $H(\div;\OmT) \setminus \Sh$. To this end we introduce two norms: One mimics an $H(\div)$-type norm where the $L^2$-part is replaced by $a_h(\cdot,\cdot)$. This norm is supposed to be applied only on the discrete space $\Sh$. For $u_h \in \Sh$ we define
\begin{equation}
    \Vert u_h \Vert_{a_h}^2 := \Vert u_h\Vert_{\Omega,j}^2 + \Vert \div u_h \Vert_{\OmT}^2,\quad \text{with} \quad
    \Vert u_h\Vert_{\Omega,j}^2 = \Vert u_h \Vert_{\Omega}^2  + \gamma_u \vert u_h \vert_j^2 
\end{equation}
The second norm is equivalent to the first one on $\Sh$, but also makes sense for general (``non-discrete'') functions in $u \in H(\div;\OmT)$: \vspace*{-0.1cm}
\begin{equation}
  \Vert u \Vert_{\Sigma}^2 := \Vert u \Vert_{\Omg}^2 + \Vert \div u \Vert_{\OmT}^2 \quad \text{with} \quad \Vert u \Vert_{\Omg} := \begin{cases}
        \Vert u \Vert_{\OmT}, &\gamma_u > 0, \\ \Vert u \Vert_{\Omega}, &\gamma_u = 0.
\end{cases}
\end{equation}
We have $\Vert u \Vert_{\Sigma} \leq \Vert u \Vert_{H(\div;\OmT)}$
which turns into an identity for $\gamma_u > 0$.
That both norms are equivalent on $\Sh$, i.e.
$\Vert u_h \Vert_{a_h} \simeq \Vert u_h \Vert_\Sigma$ (and $\Vert u_h \Vert_{\Omega,j} \simeq \Vert u_h \Vert_{\Omg}$) on $\Sh$,
follows directly from \eqref{eq:gp}. %

\begin{remark}\label{rem:weaknorms}
  Often, for the analysis of the mixed Poisson problem in a geometrically fitted setting a stronger, $H^1$-type, norm on $Q_h$ and a weaker, $L^2$-type, norm on $\Sh$ is used. To achieve this partial integration on the $b_h(\cdot,\cdot)$-term is applied in the stability (and continuity) analysis. We will not do this here as after partial integration on $b_h(\cdot,\cdot)$ we would need control on $\Sh$ in $L^2(\OmT)$  which we can only provide for $\gamma_u > 0$ which we don't want to restrict the analysis to. We discuss consequence of this choice and potential improvements for the setting with stronger norms and $\gamma_u > 0$ in more detail in Remark \ref{rem:strongnorms} below.
\end{remark}

\subsection{A-priori error analysis for $u - u_h$} \label{sec:analysis:apriori1}
\noindent
We start with the stability results that the method has been tailored for in the next two lemmas.
\begin{lemma}[Kernel-coercivity]
    $a_h(\cdot,\cdot)$ is coercive on $\Shz$ w.r.t.
    $\Vert \cdot \Vert_{a_h}$ and $\Vert \cdot \Vert_{\Sigma}$.
\end{lemma}

\begin{proof}
  For $u_h \in \Shz$ 
    $%
      a_h(u_h,u_h) = a_h(u_h,u_h) + \Vert \div u_h \Vert_{\OmT}^2 = \Vert u_h \Vert^2_{a_h} \simeq \Vert u_h \Vert^2_{\Sigma}. 
    $%
\end{proof}

\begin{lemma}[LBB-type inf-sup stability on $\OmT$]\label{lemma:InfSupBh}
    There holds the inf-sup condition
    \begin{equation}\label{eq:InfSupBh}
        \inf_{\bar p_h \in Q_h} \sup_{u_h \in \Sh} \frac{b_h(u_h,\bar p_h)}{\Vert u_h \Vert_\Sigma \Vert \bar p_h \Vert_{\OmT}} 
        \ge c > 0
    \end{equation}
    for a constant $c$ that is independent of $h$ and local cut configurations. Hence, the saddle-point problem \eqref{eq:new:mixedprob:a} -- \eqref{eq:new:mixedprob:b} is well-posed (in the corresponding norms).
\end{lemma}
\begin{proof}
    The proof follows standard arguments, cf. \cite[Section 7.1.2]{brezzi2012mixed}. Slight adaptations are necessary to account for the fact that the domain $\OmT$ depends on $h$.  For completeness we add a complete proof in \cref{app:lbb}.
\end{proof}
Both stability results together with the approximation of the r.h.s. yield a first quasi-best approximation estimate for the discretization error.
\begin{lemma}[Error estimate for $u_h$]\label{lemma:EEuhPerp}
    Let $(u,p) \in \Sigma \times Q$ be the solution  to \eqref{eq:weakForm:I} --
    \eqref{eq:weakForm:III} and $(u_h,\bar{p}_h) \in \Sh \times Q_h$ the
    discrete solution to \eqref{eq:new:mixedprob:I} --
    \eqref{eq:new:mixedprob:III}. Then, there holds 
    \begin{align} %
      \!\!\Vert u  \!-\! u_h \Vert_{\Omega} \!+\! \gamma_u^{\frac12} \vert u_h \vert_j \!+\! \Vert \bar{p}_h \!-\! \mathcal{E}_h^0 p \Vert_{\OmT}  &\! \!\lesssim \!\! \inf_{v_h \in \Sh} \!\!\Vert u \!-\! v_h \Vert_{\Omega} \! + \! \vert v_h \vert_j \!+\! \Vert \div v_h \!\!-\! f_h \Vert_{\OmT}\!.\!\!
      \intertext{ If $\gamma_u > 0$ we further have}
      \!\!\Vert \ue \!\!\!\!-\! u_h \Vert_{\OmT} \!+\! \gamma_u^{\frac12} \vert u_h \vert_j \!+\! \Vert \bar{p}_h \!-\! \mathcal{E}_h^0 p \Vert_{\OmT}  &\! \!\lesssim \!\! \inf_{v_h \in \Sh} \!\!\Vert \ue \!\!\!\!-\! v_h \Vert_{\OmT} \! + \! \vert v_h \vert_j \!+\! \Vert \div v_h \!\!-\! f_h \Vert_{\OmT}\!,\!\!
    \end{align}
    where $\ue$ is the
    Sobolev extension of $u$ (with $\ue = \nabla \pe = \nabla (\mathcal{E} p)$). 
\end{lemma}
\begin{proof}
    \begin{subequations}
    We will first bound the $p$-error in terms of the $u$-error:
    Let $v_h \in \Sh$ be arbitrarily. From the inf-sup condition in the previous lemma we know that there is a $w_h^\ast \in \Shperp$ so that
    \begin{align}
  &    \Vert \bar{p}_h - \mathcal{E}_h^0 p \Vert_{\OmT} {\Vert w_h^\ast \Vert_{\Sigma}}  \lesssim {b_h(w_h^\ast,\bar{p}_h - \mathcal{E}_h^0 p)}\stackrel{}{=} {b_h(w_h^\ast,\bar{p}_h) - b(w_h^\ast,p)} \nonumber \\[-1ex]
 & \qquad \stackrel{\eqref{eq:consistency}}{=} { a(u,w_h^\ast) - a_h(u_h,w_h^\ast)} \lesssim { a( u - u_h,w_h^\ast) - \gamma_u j_h(u_h,w_h^\ast)} \nonumber\\
& \qquad = \! { a( u - v_h,w_h^\ast) \!+\! a( v_h - u_h,w_h^\ast) + \!\gamma_u j_h(v_h-u_h,w_h^\ast) - \!\gamma_u j_h(v_h,w_h^\ast)} \nonumber \\
 &     \Longrightarrow
      \Vert \bar{p}_h \!-\! \mathcal{E}_h^0 p \Vert_{\OmT}  \lesssim \Vert u - v_h \Vert_{\Omega} + \gamma_u^{\frac12}\vert v_h \vert_j  + \Vert v_h - u_h \Vert_{\Omega,j} \label{eq:phbarbound}
    \end{align}
    For the latter term, the discrete error, we have with $w_h \coloneqq v_h-u_h$ that there holds %
    \begin{align}
        \Vert w_h & \Vert_{\Omega,j}^2 = 
        a(v_h - u_h,w_h) + \gamma_u j_h(v_h - u_h,w_h) 
        \nonumber\\
        & = a(u - u_h,w_h) - \gamma_u j_h(u_h,w_h) + a(v_h - u,w_h) + \gamma_u j_h(v_h,w_h) \nonumber \\[-1ex]
        & \!\! \stackrel{\eqref{eq:consistency}}{=} b_h(w_h, \bar{p}_h - \mathcal{E}_h^0 p) + a(v_h - u,w_h) + \gamma_u j_h(v_h,w_h)  \nonumber \\
        \nonumber & \leq \nrm{\div w_h}_{\OmT} \Vert \bar{p}_h - \mathcal{E}_h^0 p \Vert_{\OmT} + \sfrac12 \nrm{w_h}_{\Omega,j}^2 + \sfrac12 ~ (\nrm{v_h - u}_{\Omega}^2 + \gamma_u |v_h|_j^2) \\
        \nonumber & \stackrel{}{\leq} \eta \Vert \bar{p}_h - \mathcal{E}_h^0 p \Vert_{\OmT}^2 + \eta^{-1} \nrm{\div w_h}_{\OmT}^2 + \nrm{v_h - u}_{\Omega}^2 + \gamma_u |v_h|_j^2
    \end{align}
    for any $\eta > 0$, where we used a kickback argument and Young's inequality in the last step. Now using \eqref{eq:phbarbound} and choosing $\eta>0$ sufficiently small to absorb $\nrm{w_h}_{\Omega,j}$ with another kickback argument, we obtain 
    \begin{align} 
        \norm{v_h - u_h }_{\Omega,j} \lesssim  \nrm{u - v_h}_{\Omega} + \gamma_u^{\frac12} |v_h|_j + \norm{\div (v_h - u_h) }_{\OmT}
        \label{eq:intermediate:b}
    \end{align}
    Hence, we have
    \begin{align*}
        \nrm{\ue \!-\! u_h}_{\Omega} + \gamma_u^\frac12 |u_h|_j 
        &\lesssim
        \nrm{\ue \!-\! v_h}_{\Omega} + \gamma_u^\frac12 |v_h|_j + \nrm{v_h \!-\! u_h}_{\Omega,j} +  \nrm{\div (v_h \!-\! u_h)}_{\OmT} \\[-1ex]
        &\!\!\!\stackrel{\eqref{eq:intermediate:b}}{\lesssim}
        \nrm{\ue \!-\! v_h}_{\Omega} + \gamma_u^\frac12 |v_h|_j + \nrm{\div v_h \!-\! f_h}_{\OmT},
    \end{align*}
    where we made use of $\div u_h = f_h$.
    Now, let's turn to the case $\gamma_u > 0$. We have
    \begin{align*}
      \Vert \ue - u_h \Vert_{\OmT} & \leq \Vert \ue - v_h \Vert_{\OmT} + \Vert v_h - u_h \Vert_{\OmT} \lesssim \Vert \ue - v_h \Vert_{\OmT} + \Vert v_h - u_h \Vert_{\Omega,j} \\ 
      & \stackrel{\eqref{eq:intermediate:b}}{\lesssim} \Vert \ue - v_h \Vert_{\OmT} + \gamma_u^\frac12 |v_h|_j + \Vert \div v_h - f_h \Vert_{\OmT}.
    \end{align*}
\end{subequations}
\end{proof}

Note that the last term in \cref{eq:intermediate:b} vanishes in geometrically fitted methods when restricting the approximation space for $v_h$ to those fulfilling $\div v_h = \Pi^Q \div u = \Pi^Q f$ there, cf. \cref{rem:freg}. Restricting $v_h \in \Sh$ in the infimum on the r.h.s. to comply to $\div v_h = \div u_h = f_h$ would allow us to obtain a similar result here as well. However, the approximation error term $\nrm{f-f_h}_{\OmT}$ would then enter the approximation problem, which would lead to the same overall convergence result. 

Next, we take a look at the approximation problem.

\begin{theorem}\label{thm:errorest1}
  There holds for $u\in H^{m}(\Omega)$ with integer $m \in \{0,..,k+1\}$
  \begin{subequations} \label{eq:errorest1a}
  \begin{align}
    \! \Vert u - u_h \Vert_{\Omega} + \gamma_u^{\frac12} & |u_h|_j + \Vert \bar{p}_h \!- \mathcal{E}_h^0 p \Vert_{\OmT} 
    \lesssim  h^{m} \Vert u \Vert_{\!H^{m}\!(\Omega)} \! + \nrm{\Pi^Q \fe \!-\! f_h}_{\OmT}\!\!. \label{eq:errorest1a}\\
    \intertext{For $\gamma_u > 0$ there further holds}
    \! \Vert \ue - u_h \Vert_{\OmT} + \gamma_u^{\frac12} & |u_h|_j + \Vert \bar{p}_h \!- \mathcal{E}_h^0 p \Vert_{\OmT} 
    \lesssim  h^{m} \Vert u \Vert_{\!H^{m}\!(\Omega)} \! + \nrm{\Pi^Q \fe \!-\! f_h}_{\OmT}\!\!. \label{eq:errorest1b}
  \end{align}
\end{subequations}
\end{theorem}
\begin{proof}
  Starting from the quasi-best approximation result in \cref{lemma:EEuhPerp} we only need to consider the approximation problem. Here, we set $v_h = \PiRT \ue$ and directly obtain the following standard bound from the literature, cf. e.g. \cite{brezzi2012mixed}, 
  $$
  \Vert u - \PiRT \ue \Vert_{\Omega}  \leq \Vert \ue - \PiRT \ue \Vert_{\OmT} \lesssim h^{m} \Vert \ue
  \Vert_{H^{m}(\OmT)} 
  $$
  For the GP part we recall \eqref{eq:gp2} and we apply results from the literature, cf. e.g. \cite[Lemma 5.8]{LO_ESAIM_2019}, for the Lagrange interpolation operator $\Pi^{\mathcal{L}}$ of order $k$:
  \begin{align*}
  | \PiRT \ue |_j & \leq  |\Pi^{\mathcal{L}} \ue |_j + |(\Pi^{\mathcal{L}} - \id) \ue |_j + |(\PiRT - \id) \ue |_j  \\
  & \lesssim |\Pi^{\mathcal{L}} \ue |_j + \Vert(\Pi^{\mathcal{L}} - \id) \ue \Vert_{\OmT} + \Vert (\PiRT - \id) \ue \Vert_{\OmT} \lesssim h^{m} \Vert \ue \Vert_{H^{m}(\OmT)}.
  \end{align*}
  Exploiting $\Vert \ue \Vert_{H^{m}(\OmT)} \lesssim \Vert u
  \Vert_{H^{m}(\Omega)}$ and $\div \PiRT \ue = \Pi^Q \fe$ concludes the proof.
\end{proof}

\begin{remark} \label{rem:analysis:trimmed}
    Let us assume $f = f_h \in Q_h$ and $\gamma_u = 0$ and compare the solution $(u_h^R,p_h^R)$ of restricted mixed method \eqref{eq:mixedprob} and the solution $(u_h^M,\bar p_h^M)$ of the unfitted mixed method \eqref{eq:new:mixedprob:a} -- \eqref{eq:new:mixedprob:b}.
    As $\ker b = \ker b_h$ and $a_h(\cdot,\cdot) = a(\cdot,\cdot)$ we have that $\Shz$ in \eqref{eq:mixedprob:I} and \eqref{eq:new:mixedprob:I} coincide. Further, \eqref{eq:mixedprob:II} and \eqref{eq:new:mixedprob:II} yield the same solution as both lead to pointwise conditions, $\div u_h^\perp = f = f_h$ on $\Omega$ or $\OmT$, respectively.
    Hence, we have coinciding $u$-components $u_h^R = u_h^M$.
    Finally, we can rewrite \eqref{eq:mixedprob:III} with the help of $\bar p_h^R = \mathcal{E}_h^0 p_h^R$ to 
    $
b_h(v_h^\perp, \bar p_h^R) 
= b(v_h^\perp, p_h^R) = 
(v_h^\perp \cdot n,p_D)_{\Gamma} - a(u_h^{R},v_h^\perp),  v_h^\perp \in \Shperp
    $
    which coincides with \eqref{eq:new:mixedprob:III}
    and yields $\bar p_h^R = \bar p_h^M$.
  \end{remark}

\begin{lemma}\label{lem:dual}
    Assume that the domain $\Omega$ is smooth enough to provide an
    $L^2(\Omega)$-$H^2(\Omega)$ stability result (e.g. $\Omega$ is convex),     
    there holds
    $$
    \Vert \bar p_h -  \mathcal{E}_h^0 p \Vert_{\OmT}
    \lesssim  h (\| u-u_h\|_{\Omega} + \gamma_u^{\frac12}|u_h|_{j_h} ) + \|f - f_h\|_{-2}. %
    $$

  \end{lemma}
  \begin{proof}
  \begin{subequations}
  The structure of the proof follows the standard strategy of exploiting
    approximation and regularity of the dual problem. The dual problem ist considered on
    $\Omega$, followed by an extension of the solution to a larger domain. 
  
    Let us start with formulating the dual problem which is \eqref{eq:weakForm:a}--\eqref{eq:weakForm:b} with data $p_D = 0$ and $f= -(\bar p_h - \mathcal{E}_h^0 p)$ and we denote the solution as 
    $ (w,z)\in\Sigma \times Q $. 
    \begin{align}
        (w,w')_{\Omega} + (\div w', z)_{\Omega}
        &= 0
        && \forall w' \in \Sigma\tag{D1} \label{eq:ED1}\\
        (\div w, z')_{\Omega}
        &= (  (\bar p_h - \mathcal{E}_h^0 p),z')_{\Omega}
        && \forall z' \in Q  \tag{D2} \label{eq:ED2}.
      \end{align}
    There holds
    $
    \|w\|_{H^1(\Omega)} \lesssim \|z\|_{H^2(\Omega)} \lesssim \| (\bar p_h - \mathcal{E}_h^0 p)\|_{\Omega}
    $
    due to the $L^2$-$H^2$ regularity of $\Omega$. 
    Further choosing $w' = u-u_h$ in \eqref{eq:ED1} we have 
    \begin{equation} \label{eq:d1b}
        (w, u-u_h)_{\Omega} = - (\div (u-u_h), z)_{\Omega}.
    \end{equation}
        We will need a proper extension of $w$ which we construct next.
    Let $\Omega_e$ be an $h$-independent extension strip domain to $\Omega$ such that $\OmT \subset \Omega_e \cup \Omega$ and 
    $w_e, q_e\in H^1(\Omega_e)\times L^2(\Omega_e)$ be the solution of the Stokes problem on $\Omega_e$:
    \begin{align*}
        -\Delta w_e + \nabla q_e 
        &= 0,
        & 
        \div w_e
        &=  \chi_{\OmT} (\bar p_h - \mathcal{E}_h^0 p)
        && \text{ on } \Omega_e,\\
        w_e|_{\Gamma} &= w|_{\Gamma}, \text{ on } \Gamma
        & 
        w_e|_{\Gamma_e \setminus \Gamma } &= \bn \frac{\int_{\Omega_e\setminus \Omega } \chi_{\OmT} (\bar p_h - \mathcal{E}_h^0 p) }{|\Gamma_e \setminus \Gamma|}
        && \text{ on } \Gamma_e \setminus \Gamma.
    \end{align*}
    With constants that are independent of $\OmT$ we have 
    \[
        \|w_e\|_{H^1(\Omega_e)}\lesssim \|w\|_{H^1(\Omega)} + \|(\bar p_h - \mathcal{E}_h^0 p)\|_{\OmT} \lesssim \|(\bar p_h - \mathcal{E}_h^0 p)\|_{\OmT},   
    \]
    We now consider the extension of $w$ by $w_e$: 
    $\tilde w = \chi_\Omega w + \chi_{\Omega_e} w_e \in H^1(\Omega \cap \Omega_2)$
    and define 
    $w_h = \Pi^{\mathbb{R}\mathbb{T}} \tilde w|_{\OmT} \in \Sh$, for which we have the standard estimate
    \begin{equation}\label{eq:estw}
        \Vert w - w_h\Vert_{\Omega} + \gamma_u^{\frac12} \vert w_h \vert_{j_h}\leq
        \Vert \tilde w - w_h\Vert_{\OmT} \lesssim 
    h         \| \tilde w\|_{H^1(\OmT)} 
    \lesssim h \|(\bar p_h - \mathcal{E}_h^0 p)\|_{\OmT}
    \end{equation}
    and the Fortin property $\div w_h = \chi_{\OmT} (\bar p_h - \mathcal{E}_h^0 p) $.

    Subtracting \eqref{eq:new:mixedprob:a} with $v_h = w_h$ from \eqref{eq:weakForm:a} with $v = v_h = w_h \in \Sh$ and rewriting $b(w_h,p) = (\div w_h,\mathcal{E}_h^0 p)_{\OmT}$ yields
    \begin{align*}
         (u - u_h, w_h)_\Omega - \gamma_u j_h(u_h,w_h)
      & = (\div w_h, \bar p_h - \mathcal{E}_h^0 p)_{\OmT} = \Vert \bar p_h -
      \mathcal{E}_h^0 p \Vert_{\OmT}^2
    \end{align*}
  Subtracting \eqref{eq:d1b}, and exploiting $w=\tilde w|_{\Omega}$, we get
  \begin{equation}
    \Vert \bar p_h -
    \mathcal{E}_h^0 p \Vert_{\OmT}^2 = 
      (u - u_h, w_h- w)_\Omega - \gamma_u j_h(u_h,w_h) + ( \div (u-u_h), z)_\Omega.  \label{eq:dual:helpB}
    \end{equation}
    Therefore  with $\div (u-u_h) = f_h - f$ we have
    \begin{align*}
      \Vert \bar p_h -
      \mathcal{E}_h^0 p \Vert_{\OmT}^2
      \leq &
      (\| u-u_h\|_{\Omega} + \gamma_u^{\frac12}|u_h|_{j_h}) 
   (\| w-w_h\|_{\Omega} + \gamma_u^{\frac12}|w_h|_{j_h}) \\ 
   & + \| z\|_{H^2(\Omega)} \|f - f_h\|_{-2}. %
    \end{align*}
    Combining this with \eqref{eq:estw} we get
    \begin{align*}
      \Vert \bar p_h -  \mathcal{E}_h^0 p \Vert_{\OmT}^2
      & \lesssim (h (\| u-u_h\|_{\Omega} + \gamma_u^{\frac12}|u_h|_{j_h} ) + \|f - f_h\|_{-2}) \cdot \Vert \bar p_h -  \mathcal{E}_h^0 p \Vert_{\OmT}
    \end{align*}
    Dividing by $\Vert \bar p_h -  \mathcal{E}_h^0 p \Vert_{\OmT}$ yields the claim.
  \end{subequations}
  \end{proof}

  \begin{corollary}
    Let $(u,p) \in H^{\ell+1}(\Omega) \times H^{\ell}(\Omega)$ be the solution to \eqref{eq:weakForm:a}--\eqref{eq:weakForm:b} with $f \in H^{\ell+1}(\Omega)$, $0\leq \ell \leq k$, $\Omega$ be so that $L^2$-$H^2$-regularity holds and let \cref{ass:fh} hold for $k_f \geq \ell$ and $r = \ell+1$. Then there holds
    \begin{subequations}
    \begin{align}
        \| \ue -u_h\|_{\Sigma} + \gamma_u \vert u_h \vert_j & \lesssim h^{\ell+1} \Vert u \Vert_{H^{\ell+1}(\Omega)} + h^{\ell+1} \Vert f \Vert_{H^{\ell+1}(\Omega)}, \\
      \Vert \bar p_h -  \mathcal{E}_h^0 p \Vert_{\OmT} 
      & \lesssim h^{\ell+2} \Vert u \Vert_{H^{\ell+1}(\Omega)} + h^{\ell+2} \Vert f \Vert_{H^{\ell+1}(\Omega)} \label{eq:errorestp}.
    \end{align}
\end{subequations}
\end{corollary}
We note that with $\Vert p - \bar p_h \Vert_{\Omi} \leq \Vert \Pi^Q p - \bar p_h \Vert_{\Omi} +\Vert p - \Pi^Q p \Vert_{\Omi} $ \eqref{eq:errorestp} only implies the $ \mathcal{O}(h^{\ell+2})$-bound for the difference to the $L^2$ projection of $p$ onto $Q_h|_\Omi$ while $\Vert p - \Pi^Q p \Vert_{\Omi}$ is 
bounded by $h^{\ell+1} \Vert p \Vert_{H^{\ell+1}(\Omega)}$ for $\ell \leq k$. To benefit from the higher order bound for $\Vert \bar p_h -  \mathcal{E}_h^0 p \Vert_{\OmT}$ we apply post-processings as will be discussed in the next section.

\begin{remark}\label{rem:strongnorms}
  In \cref{rem:weaknorms} we commented on an alternative choice of norms for the error estimates that is often considered in the geometrically fitted case.
  As a consequence of the choice of norms chosen here, we obtain results for the error in $\bar p_h$ in a weaker norm than usual possible (in the geometrically fitted setting). However, we consider $\bar p_h$ as an intermediate variable only and regard the post-processed pressure $p_h^\ast$, introduced in the next section, \cref{sec:post-process}, as the more important approximation to $p$ which does not suffer from this potential suboptimality. More interesting is the fact that the dependency on the approximation error of $f$ could benefit from a change of norms. %
  In the numerical examples in \cref{sec:numex:f} we investigate the dependency of the solution accuracy on the approximation of $f$ and observe a dependency on the approximation error of $f$ that seems to be weaker than the one predicted by the previous error analysis. Whether the dependency on the approximation error of $f$ in the error analysis can be improved by a change of norms is left for future work.
\end{remark}

  \section{Post-processings} \label{sec:post-process}
  \noindent A common approach to achieve higher order convergence in mixed finite element methods are post-processing schemes, e.g. as in \cite{Stenberg91}, which exploit the accuracy of $u_h$ by making use of $\nabla p = u$.
  For the unfitted mixed scheme \eqref{eq:new:mixedprob:a}--\eqref{eq:new:mixedprob:b}, these schemes have the additional potential of repairing the inconsistency of $\bar p_h$ on cut elements. We propose two of such schemes.
  
  \subsection{Element-local post-processing} \label{sec:elpost-process}
  \noindent First, we introduce the following element-local post-processing scheme: 
  For each $T \in \Th$, find $p_h^* \in \mathcal{P}^{k+1}(T)$ such that 
  \begin{subequations} \label{eq:PP1}
      \begin{align}
          (\nabla p_h^*, \nabla q_h^*)_T &= (u_h,\nabla q_h^*)_T && \forall q_h^* \in \mathcal{P}^{k+1}(T)%
, \label{eq:PP1a} \tag{$T$-PP-a}\\
          (p_h^*,1)_T &= (\bar p_h,1)_T && \text{ if } T \in \Thi, \label{eq:PP1b} \tag{$T$-PP-b}\\
          (p_h^*,1)_{T \cap \Gamma} &= (p_D,1)_{T \cap \Gamma} &&\text{ if } T \in \ThG. \label{eq:PP1c}
          \tag{$T$-PP-c} \end{align}
  \end{subequations}
  This yields a post-processed field $p_h^* \in \mathbb{P}^{k+1}(\Th)$. 
  Here, on each element $T$ the equation $\nabla p = u$ is used to reconstruct $p_h^*$. This requires $u_h$ to be accurate on $T \in \Th$ (not only on $T\cap\Omega$) and hence $\gamma_u > 0$ in \eqref{eq:new:mixedprob:a}. Obviously \eqref{eq:PP1a} only determines $p_h^\ast$ only up to a constant. On uncut elements the quality of $\bar p_h$ can be exploited to fix that constant, leading to \eqref{eq:PP1b} while on cut elements we make use of the knowledge of the Dirichlet data leading to \eqref{eq:PP1b}. 
  In the analysis below, we will show that we achieve higher order convergence for $p_h^*$ with this scheme.

  Despite its simplicity this post-processing has two disadvantages: First, it relies on $u_h$ to be accurate on $\OmT$ instead of only $\Omega$ which requires $\gamma_u > 0$ and second it relies on Dirichlet data. In the next section we discuss an alternative that does not have these disadvantages.

  \subsection{Patchwise post-processing} \label{sec:patchpost-process}
  \noindent 
  Now, we consider another post-processing scheme that operates on suitable patches instead of on single elements.
 
  Again, we use $u_h$ to reconstruct $p_h^*$. However, this time we formulate corresponding patch-problems w.r.t. patches as in \cref{sec:patches} and add a proper GP stabilization on each patch that is cut by the boundary. To fix the constant on each patch, we demand the mean value on the uncut elements to match with the mean value of $\bar{p}_h$ on the uncut elements, which is known to be accurate.
  
  On each patch, we find $p_h^* \in \mathbb{P}^{k+1}(\mathcal{T}_\omega)$ so that 
  \begin{subequations}
  \begin{align}
      (\nabla p_h^*, \nabla q_h^*)_{\Omega \cap \omega} + j_h^{\omega}(p_h^*,q_h^*) &= (u_h,\nabla q_h^*)_{\Omega \cap \omega} \quad \forall q_h^* \in \mathbb{P}^{k+1}(\mathcal{T}_\omega)%
,\label{eq:PP2a} \tag{$\omega$-PP-a}\\
      (p_h^*,1)_{\omega \cap \Omi} &= (\bar p_h,1)_{\omega \cap \Omi}. \label{eq:PP2b} \tag{$\omega$-PP-b}
  \end{align}
  \end{subequations}
  Here, we set
  $j_h^\omega(u_h,v_h) = \sum_{F \in \mathcal{F}_h^\omega} h^{-2} j_F(u_h,v_h)$ as
  the patch-wise GP bilinear form 
  with $j_F(u_h,v_h)$ a GP stabilization bilinear form as considered in \cref{sec:gp}, however for polynomials up to degree $k+1$ instead of only $k$.
  Note that we introduced a scaling with $h^{-2}$ here due to the fact that we need to stabilize an
  $H^1$-type operator. We define $|q_h|_{j,\omega}^2 := j_h^\omega(q_h,q_h)$, $q_h \in \mathbb{P}^{k+1}(\mathcal{T}_h)$.

  On the trivial patches $\mathcal{T}_\omega = \{T\}$ in the interior \eqref{eq:PP2a} and \eqref{eq:PP2b} coincide with \eqref{eq:PP1a} and \eqref{eq:PP1b} as in the geometrically fitted case in the literature.

  \begin{remark}[Variants]
    Alternatively to the constraint \eqref{eq:PP2b} on every patch one could also impose 
    $(p_h^*,1)_{\omega \cap \Omega} = (\bar p_h,1)_{\omega}$ as $(\bar p_h,1)_{\omega} = ( \mathcal{E}_h^0 p,1)_{\omega} + \mathcal{O}(h^{k+2}) = ( p,1)_{\omega \cap \Omega} + \mathcal{O}(h^{k+2})$. Furthermore, instead of $\mathbb{P}^{k+1}(\mathcal{T}_{\omega})$ one polynomial space $\mathcal{P}^{k+1}(\omega)$ per patch and a local formulation without ghost penalties could be used without a loss in the convergence order.
  \end{remark}
  
  \subsection{Analysis of the post-processings} \label{sec:analysis:post-proc}
  \noindent In this section, we will analyze the post-processing schemes that were introduced in \cref{sec:elpost-process,sec:patchpost-process}. We will begin by examining the patchwise scheme, which is a bit more involved to analyse as some aspects of the analysis of the elementwise scheme can be inferred from this case. 

  \subsubsection{Patchwise post-processing}
  We start with a bound for $p^{\mathcal{E}}-p_h^*$ in the $H^1$-semi-norm.
  \begin{lemma}\label{lem:PatchPPH1Est}
      Let $p_h^\ast$ be the solution to \eqref{eq:PP2a}--\eqref{eq:PP2b} and $p \in H^{k+2}(\dom)$ be the solution to \eqref{eq:weakForm:a}--\eqref{eq:weakForm:b}. There holds 
      \begin{equation*}
          \Vert \nabla (p^{\mathcal{E}}-p_h^*) \Vert_{\OmT} + |p_h^\ast|_{j,\omega} \lesssim h^{k+1} \Vert p  \Vert_{H^{k+2}(\dom)} + \Vert u - u_h \Vert_{\Omega}.
      \end{equation*}
  \end{lemma}

  \begin{proof}
    We essentially follow the analysis in \cite[Thm. 19]{Zhang20} with adjustments to account for patches and the GP mechanism. Let $\mathcal{T}_\omega \in \mathcal{C}_h$ and $q_h \in \mathbb{P}^{k+1}(\mathcal{T}_\omega)$. We set $s_h := p_h^*-q_h$ and show a GP result for its gradient $\nabla s_h$ where we define $\overline{s_h} = \vert \omega \vert^{-1} (s_h,1)_{\omega}$ as the mean value of $s_h$ on $\omega$:
    \begin{align*}
      \Vert \nabla s_h \Vert_\omega^2 & = \Vert \nabla (s_h - \overline{s_h}) \Vert_\omega^2 \lesssim 
      h^{-2} \Vert (s_h - \overline{s_h}) \Vert_\omega^2 
      \\[-1ex] & \stackrel{\eqref{eq:gp}}{\lesssim} 
      h^{-2} (\Vert (s_h - \overline{s_h}) \Vert_{\omega \cap \Omega}^2 + h^2 j_h^\omega(s_h,s_h))
      \lesssim  
      \Vert  \nabla s_h \Vert_{\omega \cap \Omega}^2 + j_h^\omega(s_h,s_h).    
    \end{align*} 
Hence, we have
      \begin{align*}
       \Vert \nabla s_h \Vert_{\omega}^2 + j_h^{\omega}(s_h,s_h) & \lesssim \Vert \nabla s_h \Vert_{\dom \cap \omega}^2  + j_h^{\omega}(s_h,s_h) \\ & \!=\! (\nabla(p_h^*-q_h),\nabla s_h)_{\dom \cap \omega} + j_h^{\omega}(p_h^*-q_h,s_h) \\
          &\!\!\overset{\eqref{eq:PP2a}}{=}\!\! (u_h,\nabla s_h)_{\dom \cap \omega} - j_h^\omega(q_h,s_h) - (\nabla q_h,\nabla s_h)_{\dom \cap \omega} \\
          &\!\!=\! (\nabla (p-q_h),\nabla s_h)_{\dom \cap \omega} + (u_h-u,\nabla s_h)_{\dom \cap \omega} - j_h^\omega(q_h,s_h),
      \end{align*}
      where the last step follows since $u=\nabla p$. An application of Cauchy-Schwarz and division by $(\Vert \nabla s_h \Vert_{\Omega \cap \omega}^2 + j_h^{\omega}(s_h,s_h)^{\frac12}$ yields 
      \begin{align*}
        \Vert \nabla s_h \Vert_{\omega} + |s_h|_{j,\omega} \leq \Vert \nabla (p-q_h) \Vert_{\dom \cap \omega} + |q_h|_j+ \Vert u_h-u \Vert_{\dom \cap \omega}.
      \end{align*}
      With a triangle inequality we obtain 
      \begin{align*}
          \Vert \nabla(\pe-p_h^*)\Vert_{\omega} + |p_h^\ast|_{j,\omega} &\lesssim \Vert \nabla (\pe-q_h) \Vert_{\omega} + \Vert \nabla s_h \Vert_{\omega} + |s_h|_{j,\omega} + |q_h|_{j,\omega} \\
          &\lesssim \Vert \nabla (\pe-q_h) \Vert_{\omega} + |q_h|_j+ \Vert u_h - u \Vert_{\dom \cap \omega}.
      \end{align*}
      Summing over all patches $\mathcal{T}_\omega \in \mathcal{C}_h$ and using standard interpolation results yields the claim.
  \end{proof}
  
  \begin{lemma}
    Let $p_h^\ast$ be the solution to \eqref{eq:PP2a}--\eqref{eq:PP2b} and $p \in H^{k+2}(\dom)$ be the solution to \eqref{eq:weakForm:a}--\eqref{eq:weakForm:b}. There holds 
      \begin{equation}
          \Vert \pe - p_h^* \Vert_{\OmT} \lesssim h^{k+2} \Vert p \Vert_{H^{k+2}(\dom)} + \Vert \mathcal{E}_h^0 p - \bar p_h \Vert_{\Omi} + h \Vert u - u_h \Vert_\Omega.
      \end{equation}
  \end{lemma}
  \begin{proof}
    For $\mathcal{T}_\omega \in \mathcal{C}_h$, we introduce the projection $Q_\omega$ defined by
      \begin{equation*}
        Q_\omega \pe = Q_\omega p := \vert T_\omega \vert^{-1} (p,1)_{T_\omega}, \text{ where } T_\omega \in \Thi \cap \mathcal{T}_\omega \text{ is the root element}.
      \end{equation*}
      The triangle inequality gives 
      \begin{equation*}
          \Vert \pe - p_h^* \Vert_{\omega} \le \Vert (I-Q_\omega)(\pe-p_h^*) \Vert_{\omega} + \Vert Q_\omega (\pe-p_h^*) \Vert_{\omega} = I + II.
      \end{equation*}
      With $e_p' := (I-Q_\omega)(\pe-p_h^*)$ there is $Q_\omega e_p' = 0$ and we can apply a DG version of the Poincaré inequality on the patch to the first term. Using $\nabla Q_\omega(\pe-p_h^*) = 0$ we obtain that
      \begin{align*}
        I^2 =  \Vert e_p' \Vert_{\omega}^2 %
          &\lesssim h^2 \vert \nabla e_p' \vert_{\omega}^2  \!+ \!\!\sum_{F\in\mathcal{F}_h^\omega} h \nrm{\jump{e_p'}}_F^2 
          \lesssim h^2 \vert \nabla e_p' \vert_{\omega}^2  \!+ \!\!\sum_{F\in\mathcal{F}_h^\omega} h \nrm{\jump{p_h^\ast}}_F^2 \\
          & \lesssim h^2 (\vert \pe  -p_h^* \vert_{H^1(\omega)}^2 \!+\!\!
          |p_h^\ast|_{j,\omega}^2 )
  \end{align*}
  where we exploited that for arbitrary $v_h \in \mathcal{P}^{k+1}(\omega_F)$ there holds $\jump{v_h} = 0$ on $F \in \mathcal{F}_h^{\omega}$ and hence \vspace*{-0.3cm}
  $$
  h \Vert \jump{p_h^\ast} \Vert_{F}^2 = \!\!\inf_{v_h \in \mathcal{P}^{k+1}(\omega_F)}  \!\! h \Vert \jump{p_h^\ast - v_h} \Vert_{F}^2 \lesssim \!\! \inf_{v_h \in \mathcal{P}^{k+1}(\omega_F)} \!\! \Vert p_h^\ast - v_h \Vert_{\omega_F}^2 \!\! \stackrel{\eqref{eq:gp0}}{\lesssim} \!\! h^2 j_F^{\omega}(p_h^\ast,p_h^\ast)
  $$
  From standard interpolation results and \cref{lem:PatchPPH1Est} there holds
  \begin{align*}
    I = \Vert e_p' \Vert_{\omega}    &\lesssim h^{k+2} \Vert \pe \Vert_{H^{k+2}( \omega)} + h \Vert u - u_h \Vert_{\omega \cap \Omega},
      \end{align*}
      To consider the second term, we apply the constraint equation \eqref{eq:PP2b} of the post-processing scheme. Let $T_\omega \in \mathcal{T}_\omega \cap \omega$ be the root element. Then, it holds that 
      \begin{equation*}
       II =   \Vert Q_\omega(p-p_h^*) \Vert_{\dom \cap \omega} \le \Vert Q_\omega(p-p_h^*) \Vert_{T_\omega} = \Vert Q_\omega(p-\bar p_h) \Vert_{T_\omega} \le \Vert \mathcal{E}_h^0 p - \bar p_h \Vert_{T_\omega}.
      \end{equation*}
      Summing over all patches $\mathcal{T}_\omega \in \mathcal{C}_h$ yields the claim.
  \end{proof}

  We will conclude convergence rates for $p_h^\ast$ only after the analysis of the elementwise post-processing scheme in the next section.

  \subsubsection{Element-local post-processing}
  We now turn to the elementwise post-processing scheme. 
  \begin{lemma}\label{lem:ElPPH1Est}
    Assuming $\gamma_u > 0$, let $p_h^\ast$ be the solution to \eqref{eq:PP1a}--\eqref{eq:PP1c} and $p \in H^{k+2}(\dom)$ the solution to \eqref{eq:weakForm:a}--\eqref{eq:weakForm:b}. There holds 
      \begin{equation} \label{eq:ElPPH1Est}
          \Vert \nabla(\pe - p_h^*) \Vert_{\OmT} \lesssim h^{k+1} \Vert p \Vert_{H^{k+2}(\Omega)} + \Vert u - u_h \Vert_{\OmT}
      \end{equation}
  \end{lemma}
  \begin{proof}
    This follows as in \cref{lem:PatchPPH1Est} with $T_\omega \in \mathcal{C}_h$ and $\Omega$ replaced by $T \in \mathcal{T}_h$ and $\OmT$, respectively, $j_h^\omega(u_h,v_h) = 0$ vanishing and $\Vert p^{\mathcal{E}} \Vert_{H^{k+2}(\OmT)} \lesssim \Vert p \Vert_{H^{k+2}(\Omega)}$.
  \end{proof}
  Note that the r.h.s. in \eqref{eq:ElPPH1Est} depends on the $L^2$-norm of the $u$-error measured on $\OmT$ and thus relies on $\gamma_u > 0$.

  \begin{lemma}
    Assuming $\gamma_u > 0$, let $p_h^\ast$ be the solution to \eqref{eq:PP1a}--\eqref{eq:PP1c} and $p \in H^{k+2}(\dom)$ be the solution to \eqref{eq:weakForm:a}--\eqref{eq:weakForm:b}. There holds
      \begin{equation}
          \Vert \pe - p_h^* \Vert_{\OmT} \lesssim h^{k+2} \Vert p \Vert_{H^{k+2}(\Omega)} + \Vert \mathcal{E}_h^0 p - \bar p_h \Vert_{\Omi} + h \Vert u - u_h \Vert_{\OmT}.
      \end{equation}
  \end{lemma}
  \begin{proof}
      Let $T \in \mathcal{T}_h$. To incorporate the constraint equations \eqref{eq:PP1b}, \eqref{eq:PP1c}, we define
      \begin{equation*}
          Q_T p := \begin{cases}
              \frac{1}{\vert T \vert} (p,1)_T &\text{ if } T \in \Thi, \\
              \frac{1}{\vert T \cap \Gamma \vert} (p,1)_{T \cap \Gamma} &\text{ if } T \in \ThG,
          \end{cases}
      \end{equation*}
      Then, the triangle inequality gives 
      \begin{equation*}
          \Vert \pe - p_h^* \Vert_{T} \le \Vert (I-Q_T)(\pe-p_h^*) \Vert_{T} + \Vert Q_T (\pe - p_h^*) \Vert_{T}.
      \end{equation*}
      For the first term, we note that $Q_T(I-Q_T)(\pe-p_h^*)=0$. Hence, we can apply the Poincaré inequality. Using $\nabla Q_T(\pe-p_h^*)=0$ and \cref{lem:ElPPH1Est} yields
      \begin{align*}
          \Vert (I-Q_T)(\pe-p_h^*) \Vert_{T} &\lesssim h \vert (I-Q_T)(\pe-p_h^*) \vert_{H^1(T)}\\ 
          &\simeq h \vert \pe-p_h^* \vert_{H^1(T)} \lesssim h^{k+2} \Vert \pe \Vert_{H^{k+2}(T)} + h \Vert u - u_h \Vert_T,
       \end{align*}
       For the second term, we differentiate between cut and interior elements and apply the constraint equations. For $T \in \ThG$, \eqref{eq:PP1c} yields
      $
        \Vert Q_T (\pe - p_h^*) \Vert_{T} = \Vert Q_T (\pe - p_D) \Vert_{T} = 0,
      $
       since $\pe \vert_{ \partial \dom} = p_D$. Now, let $T \in \ThG$ and apply \eqref{eq:PP1b} 
       \begin{equation*}
          \Vert Q_T (\pe - p_h^*) \Vert_{T} = \Vert Q_T (\pe - \bar p_h) \Vert_{T} \le \Vert \mathcal{E}_h^0 p - \bar p_h \Vert_{T}.
       \end{equation*}
       Summing over all elements $T \in \mathcal{T}_h$ yields the claim.
  \end{proof}
  
\subsubsection{Error bound for post-processed solutions}
From the previous two sections, we can conclude the following error bounds for the post-processed solutions $p_h^\ast$.

\begin{corollary}
  Let $p_h^{\ast}$ be the solution to \eqref{eq:PP1a}--\eqref{eq:PP1c} and 
 $\gamma_u > 0$ or $p_h^{\ast}$ be the solution to \eqref{eq:PP2a}--\eqref{eq:PP2b}. Further, we assume $p \in H^{k+2}(\dom)$ to be the solution to \eqref{eq:weakForm:a}--\eqref{eq:weakForm:b}, \cref{ass:fh} to hold with $k_f=k$ for $f \in H^{k+1}(\Omega)$ and $L^2$-$H^2$-regularity to hold for $\Omega$. Then, there holds
  \begin{equation}
      \Vert \pe - p_h^* \Vert_{\OmT} \lesssim h^{k+2} ( \Vert p \Vert_{H^{k+2}(\Omega)} + \nrm{f}_{H^{k+1}(\Omega)} ).
  \end{equation}
\end{corollary}
Note that for $p \in H^{k+3}(\Omega)$ we can directly induce $f = - \Delta p \in H^{k+1}(\Omega)$ so that 
$\Vert \pe - p_h^* \Vert_{\OmT} \lesssim h^{k+2} \Vert p \Vert_{H^{k+3}(\Omega)}$, cf. \cref{rem:freg}.

\section{Extensions} \label{sec:further}
In this section we want to discuss two extensions of the proposed method. First, we want to discuss how Neumann-type boundary conditions can be prescribed in the given framework. This will be of specific interest when considering a generalization of this discretization approach for Dirichlet boundary conditions for the Stokes (and Stokes-type) problems. Afterwards, we will discuss an equivalent hybridized mixed formulation for the Dirichlet case that allows to improve the computational efficiency of the unfitted mixed formulation.

\subsection{Neumann-type boundary conditions} \label{sec:neumann}

Instead of $p=p_D$ on $\Gamma$ we now ask for $u\cdot n = g_N$ on $\Gamma$ and factor out the constant on $Q$ and consequently also on $Q_h$, i.e. $Q_h = \mathbb{P}^{k}(\Th)/ \mathbb{R}$. In this section we only  sketch a method that we will also investigate in the numerical examples, but  leave a complete analysis for future work.

As we want to put emphasis on the mass conservation we impose boundary conditions through a stabilized Lagrange multiplier method avoiding the interference with the mass balance. This leads to the following discrete variational formulation:
Find $(u_h,\bar p_h, \lambda_h) \in \Sh \times Q_h \times F_h$ with $F_h \coloneqq \mathbb{P}^k(\ThG)$, s.t.
\begin{subequations}
\begin{align}
  a_h(u_h,v_h) &+ b_h(v_h,\bar p_h) \!\!\!\!\!\!\!\!\!\!\!\!\!\!\!& + ~ c_h(v_h,\lambda_h) =&  ~~0 && \forall v_h \in \Sh, \label{eq:neumann:a} \tag{N-a}\\
  b_h(u_h, q_h) && =&~~ (-f_h, q_h)_{\Gamma} && \forall q_h \in Q_h, \label{eq:neumann:b} \tag{N-b} \\
  c_h(u_h,\mu_h) && ~~~-~ j_h^{\Gamma}(\lambda_h,\mu_h) =&~~ (g_N,\mu_h)_\Gamma && \forall \mu_h \in F_h, \label{eq:neumann:c} \tag{N-c}
\end{align}
\end{subequations}
with $c_h(v_h,\mu_h) \coloneqq (v_h \cdot n, \mu_h)_{\Gamma}$
and 
$$
j_h^\Gamma(\lambda_h,\mu_h) \coloneqq \sum_{F\in\mathcal{F}_h^\Gamma} \gamma^N_F h^{-1} j_F(\lambda_h,\mu_h) + \sum_{T\in\mathcal{T}_h} \gamma^N_T h (\nabla \lambda_h \!\cdot\! n_h, \nabla \mu_h \!\cdot\! n_h)_T,
$$
where
$\mathcal{F}_h^\Gamma \coloneqq \mathcal{F}_h(\ThG)$ is the set of facets within the domain of cut elements, $\gamma^N_T,~\gamma^N_F > 0$ are stabilization parameters and 
$n_h$ is a suitable quasi-normal field to $\Gamma$, i.e. a field that has $n_h|_\Gamma = n$, the outer normal to $\Gamma$, on $\Gamma$, is smooth in $\OmG$ and approximately parallel to $\nabla d$ where $d$ is the signed distance function to $\Gamma$.

Following \cite{burman2014projection}, cf. especially the discussion in \cite[Section IV.A]{burman2014projection} for the imposition of boundary conditions in fictitious domain methods, we can interpret $j_h^\Gamma(\cdot,\cdot)$ as a penalty term that penalizes the distance of the Lagrange multiplier to a subspace $F_h^{c}$ in $F_h$ which consists of functions on a coarser (agglomerated) mesh,
for which an inf-sup stability result of the form 
$$
\inf_{\mu_h \in F_h^c} \sup_{v_h \in \Sigma_h} \frac{c_h(v_h,\mu_h)}{\Vert v_h \Vert_{\Sigma} \Vert \mu_h \Vert_{\frac12,h,\Gamma}} \geq c
$$ 
for a constant $c$ independent of the mesh size and the cut configuration holds, but which is typically not known explicitly. To prove the existence of such a subspace $F_h^c$ is more involved here than in \cite[Section IV.A]{burman2014projection} (or in \cite{burman2022cut} for the Stokes problem) where only the lowest order case has been considered. Adjustments to higher order as well as the normal extension are required. We leave the analysis for future work.

From these arguments we see that $\mathcal{F}_h^\Gamma$ should optimally be chosen as the set of facets on the interior of the patches that correspond to the coarse elements of the space $F_h^c$. However, as this is not known explicitly, we choose $\mathcal{F}_h^\Gamma$ as the set of all facets that are cut by $\Gamma$ and hence include interior facets of the patches.
The rationale behind the second term of $j_h^\Gamma(\cdot,\cdot)$, which is only relevant for $k>0$, is that functions in $F_h$ are defined in the volume, but are only relevant on $\Gamma$, similar to problems faced in higher order TraceFEM, cf. also \cite{grande2018analysis,burman2018cut}. Orthogonal to $\Gamma$ this second term describes an additional condition on $\lambda_h$ in order to make it well-defined not only on $\Gamma$ but also in the volume $\OmG$. In view of the analysis sketched before, this corresponds to a coarse subspace $F_h^c$ consisting of functions that are (at least approximately) constant in normal direction.

\subsection{Hybridization}
In the case of Dirichlet boundary conditions and the case $\gamma_u = 0$, i.e. no GP stabilization is used, the 
communication between degrees of freedom takes place only through the volume elements as in the body-fitted case which allows to apply hybridization to obtain an \emph{equivalent} formulation which
allows for static condensation into a formulation that leads to a linear system with symmetric
positive definite system matrix.
To this end we introduce a new facet finite element space for the Lagrange
multiplier of the normal continuity, $\Lambda_h \coloneqq \mathbb{P}^k(\mathcal{F}_h)$ and break up the
normal-continuity of $\Sh$ leading to $\Sh^- = \mathbb{R}\mathbb{T}_-^k(\Th)$.
We formulate the hybridized mixed formulation (we use subscript $T$ instead of $h$ to emphasize the association to the volume instead of a facet): Find $(u_h,\bar p_T, \bar p_F)
\in \Sh^- \times Q_h \times \Lambda_h$ such that
    \begin{align}
      && a_h(u_h,v_h) &~+~& \hspace*{-0.75cm} b_h(v_h,\bar p_T) + c(v_h,\bar p_F) & = (v_h \cdot  n,p_D)_{\Gamma} && \forall v_h \in \Sh^-, \label{eq:new:hybridmixedprob:a} \tag{HM-a}\\
       && b_h(u_h,\bar q_T)&& &= -(f_h, \bar q_T)_{\OmT} && \forall \bar q_T \in Q_h,  \label{eq:new:hybridmixedprob:b} \tag{HM-b} \\
        && c_h(u_h,\bar q_F) && &= 0 && \forall \bar q_F \in \Lambda_h,  \label{eq:new:hybridmixedprob:c} \tag{HM-c}
    \end{align}
    with $c_h(v_h,\bar q_F) = - (\jump{v_h \cdot n}, \bar q_F)_{\mathcal{F}_h} =
    \sum_{T \in \mathcal{T}_h}
    - (v_h \cdot n, \bar q_F)_{\partial T \setminus \partial \OmG}$.

    Note that the degrees of freedoms of $\Sh^-$ and $Q_h$ are element-local
    and only depend on $\bar p_F$ (and the r.h.s.) so that they can be
    eliminated (as the local problems are uniquely solvable). This leads to a
    (coercive) formulation for $\bar p_F$ of the form $\ell_h(\bar p_F,\bar q_F) = g_h(\bar q_F)$ for all $\bar q_F \in \Lambda_h$, cf. \cite{brezzi2012mixed}. After obtaining $\bar p_F$ the remaining system for $u_h$ and $\bar p_T$ can be solved locally on each element (in parallel). 

    \begin{lemma}
      \eqref{eq:new:hybridmixedprob:a} -- \eqref{eq:new:hybridmixedprob:c} is uniquely solvable and its solution components $(u_h,\bar p_T)$ have $u_h \in \Sh$ and $\bar p_T \in Q_h$ and also solve \eqref{eq:new:mixedprob:a} -- \eqref{eq:new:mixedprob:b}.
    \end{lemma}
    \begin{proof}
      We prove that for $p_D=0$ and $f_h=0$ the only solution to \eqref{eq:new:hybridmixedprob:a} -- \eqref{eq:new:hybridmixedprob:c} is $(u_h, \bar p_T, \bar p_F) = 0$. The general case follows by linearity. With \eqref{eq:new:hybridmixedprob:c} we can set $\bar q_F = \jump{u_h} \cdot n \in \Lambda_h$ on $\mathcal{F}_h$ and obtain normal-continuity, i.e. $u_h \in \Sh$. With \eqref{eq:new:hybridmixedprob:b} we can set $\bar q_T = \div u_h \in Q_h$ and obtain $\div u_h = 0$ in $\OmT$. In \eqref{eq:new:hybridmixedprob:a} we can then set $v_h = u_h$ and then have $b_h(v_h,\bar p_T)=c_h(v_h,\bar p_F) = 0$ and it remains $a_h(u_h,u_h) = 0$ which implies $u_h = 0$. Now setting $v_h \in \Sh$ so that $\div v_h = \bar p_T$ in \eqref{eq:new:hybridmixedprob:a} we obtain $\bar p_T = 0$ and finally setting $v_h \in \Sh$ so that $\jump{v_h} = \bar p_F$\footnote{That $\div v_h = \bar p_T$ is indeed possibly for every $\bar p_T$ is a consequence of \cref{lemma:InfSupBh} while for $\jump{v_h} = \bar p_F$ we refer to the literature, cf. e.g. \cite{brezzi2012mixed}.} we conclude with $\bar p_F=0$. Now consider arbitrary r.h.s. for $p_D$ and $f_h$. Then the solution to \eqref{eq:new:hybridmixedprob:a} -- \eqref{eq:new:hybridmixedprob:c} has $u_h \in \Sh$ and by reducing the test functions in \eqref{eq:new:hybridmixedprob:a} from $\Sh^-$ to $\Sh$ we re-obtain the problem \eqref{eq:new:mixedprob:a} -- \eqref{eq:new:mixedprob:b}.
    \end{proof}

    \begin{remark}
      Hybridization is straight-forward for $\gamma_u = 0$, but can also be applied for the GP-stabilized method with $\gamma_u > 0$. To this end the Schur complements within the static condensation procedure are to be formed patch- instead of element-wise. The resulting linear system is still symmetric positive definite, but the coupling stencil of unknowns corresponding to facet functions on cut element patches will be (significantly) larger.
    \end{remark}

\subsection{External force on the flow}
The model problem \eqref{eq:mixedPoisson} can be considered as a model for a fluid flow, where we can consider an external force acting directly on the flow. This can be modeled by a non-trivial right hand side $g \in L^2(\Omega)$ in the first equation:
\[
u - \nabla p = g \quad \text{in } \Omega, \qquad  \div u = f \quad \text{in } \Omega, \qquad p = p_D \quad \text{on } \Gamma.
\] 
The changes to the analysis are minor and we will not discuss them in detail here. The main difference is that the right hand side $g$ will be included in the right hand side of the weak formulation \eqref{eq:weakForm:a} and unfitted mixed FEM \eqref{eq:new:mixedprob:a}. These now read
\begin{align}
  a(u,v) + b(v,p) &= (g,v)_{\Omega} + (v \cdot n,p_D)_{\Gamma} && \forall v \in \Sigma, \label{eq:weakForm:a-star} \tag{C-a*} \\
  a_h(u_h,v_h) + b_h(v_h,\bar p_h) &= (g,v_h)_{\Omega} + (v_h \cdot n,p_D)_{\Gamma} && \forall v_h \in \Sh, \label{eq:new:mixedprob:a-star} \tag{M-a*}
\end{align}
respectively. This problem does not introduce new inconsistencies or challenges. Unlike for $f$, there is no need to extend $g$ to $\OmT$. The consistency relation \eqref{eq:consistency} will still hold and the error analysis in Lemma \ref{lemma:EEuhPerp} remains unchanged by the additional force $g$. The further error analysis in Section \ref{sec:analysis1} also remains valid with this additional force term.

\section{Numerical examples} \label{sec:numerics}
\noindent In this section, we will apply the previously introduced methods to an example problem.
The experiments are performed with \texttt{ngsxfem} \cite{LHPvW21}, an extension to the finite element library \texttt{ngsolve} \cite{Sch97,Sch14} and are accessible and reproducible \cite{LYR2CG_2023}, see also the \hyperlink{sec:dataavail}{section on data availability} below.
To solve the arising linear system, we use sparse direct solvers.

\noindent We consider a ring geometry inside the domain $[-1,1]^2$ with inner radius $R_1 = \sfrac14$ and outer radius $R_2 = \sfrac34$ and define $r(x) = \sqrt{x_1^2+x_2^2}$, $\bar R = \sfrac12 (R_1+R_2)$ and $\Delta R = R_2-R_1$.
The geometry is described by the following signed distance function: 
\begin{equation}
    \phi(x) = |r(x) - \bar R| - \frac{\Delta R}{2}.
\end{equation}
We consider the (artificial) exact solution $p(x) = \sin(x_1)$
and prescribe the corresponding right hand side $f = -\Delta p$ and $p_D=p|_\Gamma$ (respectively $g_N = \nabla p \cdot n$ in the Neumann case). For the evaluation of the error on $\OmT$ we use the closed form presentation, but interpret $p$ as a function on $\OmT$, $p: \OmT \to \mathbb{R}$ without any further reflection in the notation. The same holds for the derived functions $u = \nabla p$ and $f = - \div u = - \Delta p$. If not addressed otherwise we define $f_h$ through \eqref{eq:fh} with $f_h \in Q_h^{k_f} = \mathbb{P}^{k_f}(\Th)$ and $k_f = k$.

\begin{figure}
    \centering
    \includegraphics[width=0.45\textwidth]{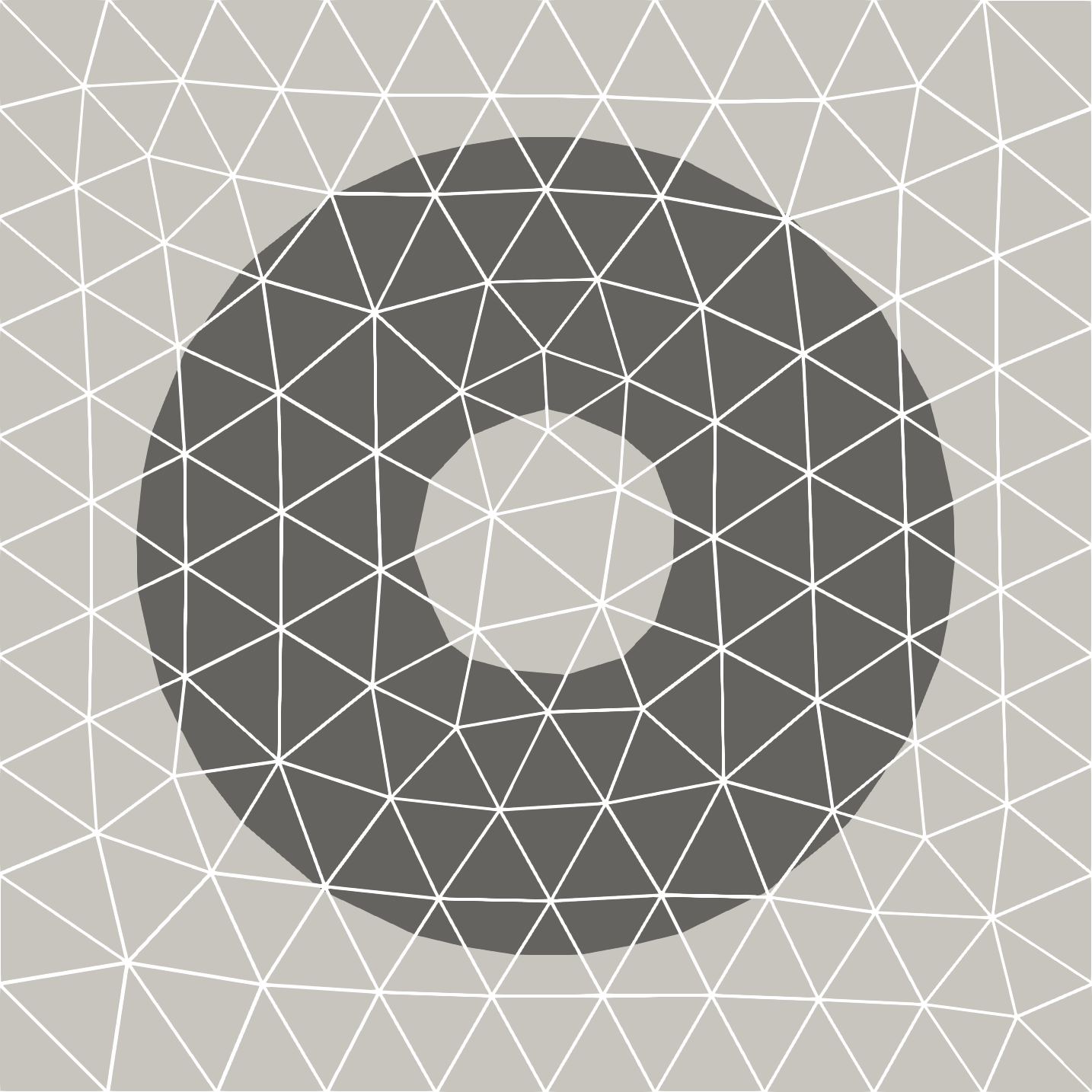}
    \includegraphics[width=0.45\textwidth]{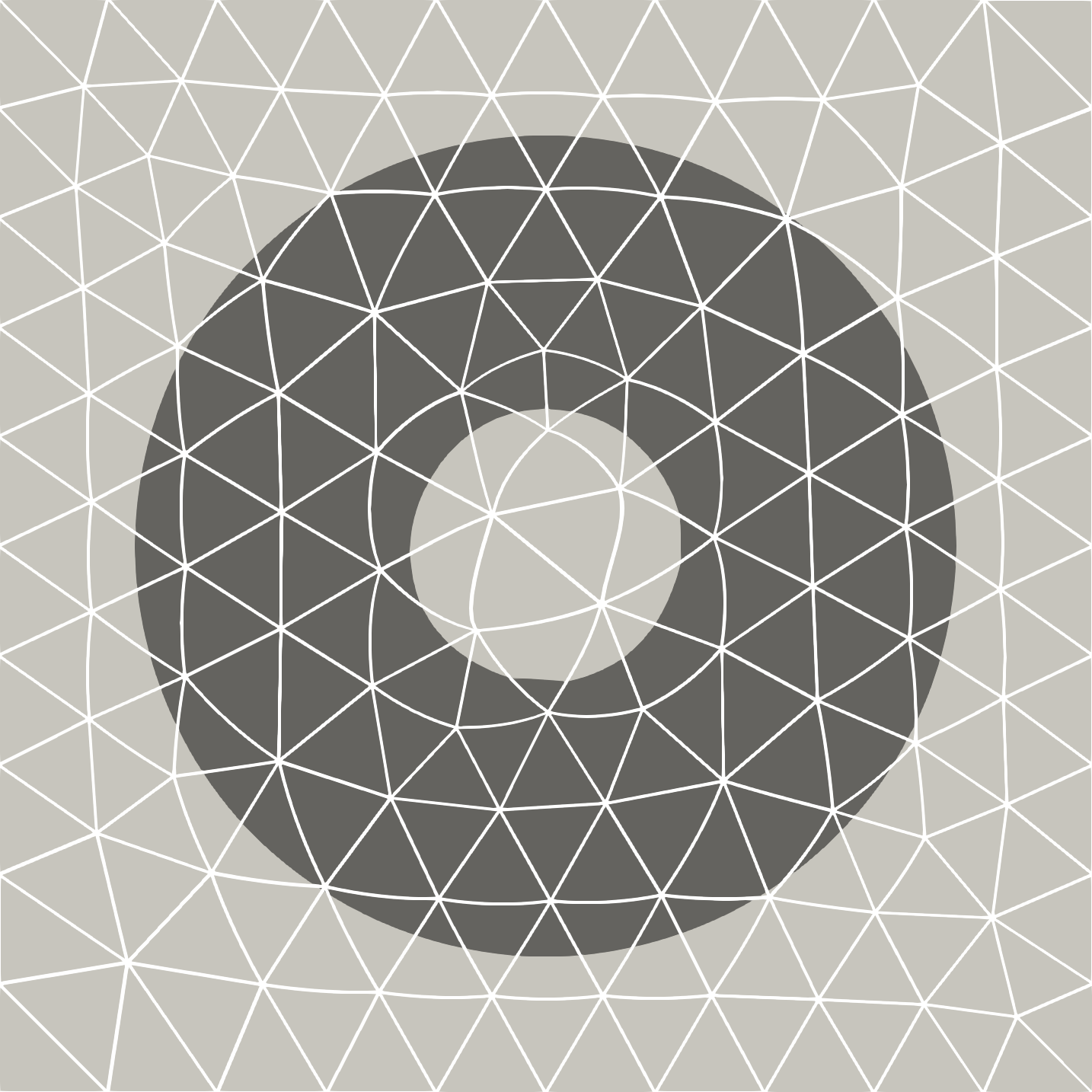} 
  \caption{Background mesh (refinement level $0$) and cut domain for the considered geometry before (left) and after (right) parametric mesh deformation (for a mesh deformation of third order).} \label{fig:meshL0}
\end{figure}

In this section we will discuss numerical studies for this problem, carried out on a sequence of suitably deformed meshes $\Th$ that are subsequently obtained from successive uniform refinements from the same background mesh with approximate mesh size $0.1$, cf. \cref{fig:meshL0}. After briefly recalling how such geometries can be handled, we will apply the unfitted mixed finite element method \eqref{eq:new:mixedprob:a}--\eqref{eq:new:mixedprob:b} and discuss the influence of the stabilization parameter $\gamma_u$ on the convergence of the method. We then discuss the convergence of the post-processed solution. Afterwards, we deviate from the cases $k_f = k$ and $p_D=p|_\Gamma$ to discuss the influence of the approximation of $f$ and to show how the method can be extended to Neumann-type boundary conditions. In the final two subsections we turn our attention to different GP stabilizations, discussing their impact on the sparsity pattern of the discretization matrix. Additionally, we compare unstabilized and stabilized methods in view of the condition number of the system matrix for a specific example.

\subsection{Geometry handling}  \label{sec:geom}
In the analysis in this work we do not consider the influence of the geometry handling on the overall error. However, an accurate handling of the unfitted geometry is crucial for the overall accuracy of the method, especially for higher order accuracy. Here, we use the isoparametric unfitted finite element method, cf. \cite{L_CMAME_2016}, where a piecewise linear approximation of the level set function $\phi(x)$ is used as a starting point. The piecewise linear approximation allows for robust numerical integration on unfitted geometries, but is only second order accurate. A mesh deformation is then computed and performed to map certain near zero iso-lines (or iso-surfaces) of the piecewise linear level set function towards corresponding iso-lines (or iso-surfaces) of the exact (or sufficiently accurate approximation of the) level set function. For details we refer to the literature, see \cite{L_CMAME_2016,LR_IMAJNA_2018}. As a consequence of the mesh deformation $\Theta_h: \widetilde \Omega \to \widetilde \Omega $, which itself is a finite element function $\Theta_h \in [\mathbb{P}^{k+1}(\widetilde{\mathcal{T}}_h) \cap H^1(\widetilde{\Omega})]^d$ we will have that $\Th$ contains curved elements and will adjust the finite element spaces accordingly. $\Theta_h^{-1}(\Th)$ is the corresponding mesh consisting of straight simplicial elements on which the finite element spaces are defined as usual, which we denote as $\hat{\Sigma}_h$, $\hat{Q}_h$, $\hat{F}_h$ and $\hat{\Lambda}_h$, respectively.
For $\Sh$ we will apply the Piola transformation, while functions in $Q_h$, $F_h$ and $\Lambda_h$ are directly mapped, i.e.
\begin{align*}
    \Sh &= \{ v = (\det D\Theta_h)^{-1} D \Theta_h \cdot \hat v ) \circ  \Theta_h^{-1}, \hat{v} \in \hat{\Sigma}_h \}, \\
    Q_h &= \hat Q_h \circ \Theta_h^{-1},
    F_h = \hat F_h \circ \Theta_h^{-1}, 
    \Lambda_h = \hat \Lambda_h \circ \Theta_h^{-1}.
\end{align*}
The Piola mapping for $\Sh$ ensures that the bilinear forms $b_h(\cdot,\cdot)$ and $c_h(\cdot,\cdot)$ (on the corresponding discrete spaces) are geometry-independent, i.e. they do not depend on $\Theta_h$, cf. \cite{monk03:_fem_maxwell}. This implies especially that there still holds normal-continuity and $|\det D \Theta_h| \div u_h \circ \Theta_h = \Pi ( |\det D \Theta_h| \cdot f_h \circ \Theta_h )$ pointwise where $\Pi$ is the $L^2$-projection onto $\hat Q_h$. Hence, $\div u_h$ is the suitably scaled $L^2$-projection of $f_h \circ \Theta_h$ onto $\hat Q_h$ which preserves the crucial property $\div u_h = 0$ pointwise for $f_h = 0$.

\subsection{Convergence of $u-u_h$ of the unfitted mixed FE methods}

\renewcommand{\g}{1}
\newcommand{\lw}{1.2pt}
\newcommand{\lsty}{dashed}

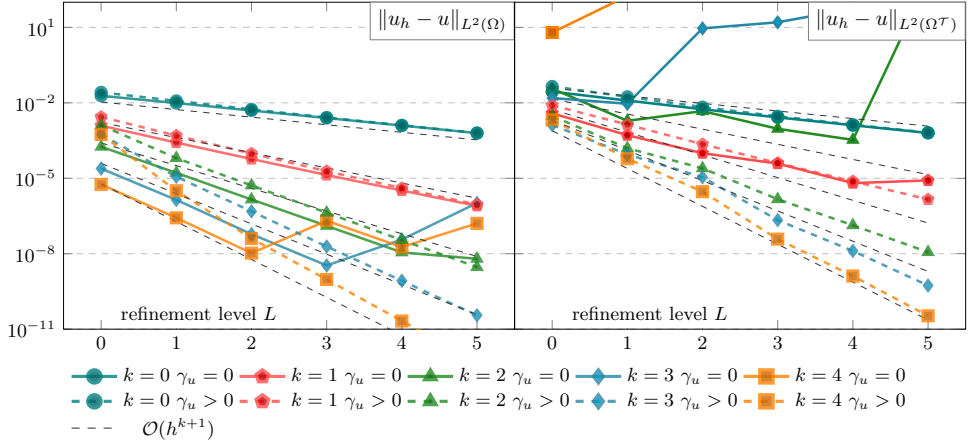
\begin{figure}[!htbp]
  \begin{center}\hspace*{-0.3cm}
      \begin{tikzpicture}[scale=0.8]
          \begin{groupplot}[%
              group style={%
              group size=2 by 1,
              horizontal sep=0cm,
              vertical sep=0.0cm,
              },
          ymajorgrids=true,
          grid style=dashed,
          ymin=1e-11, ymax=1e2,
          ]       
          \nextgroupplot[width=9cm,height=7cm,domain=0:6,xmode=linear,ymode=log, 
                         xlabel={refinement level $L$},
                         x label style={at={(axis description cs:0.3,0.09)},anchor=north,draw=none,fill=white},
                         title={$\Vert u_h - u \Vert_{L^2(\Omega)}$}, 
                         title style={at={(axis description cs:1,0.9625)},anchor=north east,draw=gray,fill=white},
                         ylabel={}, 
                         cycle list name=unfmixed1, 
                         legend style={legend columns=5, draw=none,nodes={scale=.8}}, 
                         legend to name=named,xtick={0,1,2,3,4,5,6},
          ]
          \foreach \p in {0,1,2,3,4}{
              \addplot+[discard if not={k}{\p},discard if not={gammastab}{0},discard if not={postprocessver}{0}, line width=\lw,opacity=.8] table [x=L, y=ul2error, col sep=comma] {unfmixed.csv};

          }
          \foreach \p in {0,1,2,3,4}{
              \addplot+[discard if not={k}{\p},discard if not={gammastab}{\g},discard if not={postprocessver}{0}, line width=\lw,opacity=.8,\lsty] table [x=L, y=ul2error, col sep=comma] {unfmixed.csv};
          }
          \foreach \p in {0,1,2,3,4}{
               \addplot[black!80!white, dashed, domain=0:5] {(0.07*(1/2^(\p+1))^(x+2.7))};
          }
          \legend{$k=0\ \gamma_u=0$,$k=1\ \gamma_u=0$,$k=2\ \gamma_u=0$,$k=3\ \gamma_u=0$,$k=4\ \gamma_u=0$,$k=0\ \gamma_u>0$,$k=1\ \gamma_u>0$,$k=2\ \gamma_u>0$,$k=3\ \gamma_u>0$,$k=4\ \gamma_u>0$,$\mathcal{O}(h^{k+1})$}
          \nextgroupplot[width=9cm,height=7cm,domain=0:5,xmode=linear,ymode=log, 
                         x label style={at={(axis description cs:0.3,0.09)},anchor=north,draw=none,fill=white},
                         title={$\Vert u_h - u \Vert_{L^2(\OmT)}$}, 
                         title style={at={(axis description cs:1,0.9625)},anchor=north east,draw=gray,fill=white},
                         xlabel={refinement level $L$}, 
                         ylabel={}, 
                         cycle list name=unfmixed1,
                         minor tick style ={white},
                         xtick={0,1,2,3,4,5,6},
                         yticklabels={,,}
              ]
              
              \foreach \p in {0,1,2,3,4}{
                  \addplot+[discard if not={k}{\p},discard if not={gammastab}{0},discard if not={postprocessver}{2}, line width=\lw] table [x=L, y=ul2error_bar, col sep=comma] {unfmixed.csv};
              }
              \foreach \p in {0,1,2,3,4}{
                  \addplot+[discard if not={k}{\p},discard if not={gammastab}{\g},discard if not={postprocessver}{2}, line width=\lw,opacity=.8,\lsty] table [x=L, y=ul2error_bar, col sep=comma] {unfmixed.csv};
              }
              \foreach \p in {0,1,2,3,4}{
                \addplot[black!80!white, dashed, domain=0:5] {(0.1*(1/2^(\p+1))^(x+1.4))};
           }
              
          \end{groupplot}
      \end{tikzpicture}
      \pgfplotslegendfromname{named}
  \end{center}\vspace*{-0.2cm}
  \caption{Numerical results for $\Vert u_h - u \Vert_{L^2}$ on $\dom$ and $\OmT$ with $\gamma_u = 0$ and $\gamma_u=1$ for polynomial degrees $k = 0,1,\dots,4$.} \label{fig:uConvRates}
\end{figure}

In this section, we investigate the approximation error of $u$ in the $L^2$-norm on the domains $\dom$ and $\OmT$. In both cases, we compare the case where $\gamma_u = 0$, i.e. without GP stabilization, and the case $\gamma_u = 1$, i.e. with GP stabilization.
The results are presented in \cref{fig:uConvRates}.
On $\dom$, we observe the expected convergence rates of order $k+1$ in both cases, with and without ghost penalty. The absolute error is slightly higher when $\gamma_u > 0$. However, for high accuracy computations we observe that the missing control on the condition number in the case $\gamma_u=0$ can lead to a pollution of the numerical results by accumulated round-offs errors that yield suboptimal errors so that at around $\approx 10^{-7}$ the error does not decrease anymore.
On $\OmT$, we observe that GP stabilization $\gamma_u > 0$ is necessary to obtain the expected convergence rates, at least for $k \geq 2$. With $\gamma_u > 0$ the expected convergence rates are observed. For $\gamma_u = 0$ the results are useless, unless $k \leq 1$, but even for $k=1$ the convergence rate seems to deteriorate on finer meshes. Overall these results underline the necessity for GP stabilization if accuracy on $\OmT$ is needed or computer arithmetics are not sufficient to handle the ill-conditioning of the linear system.
Finally, Fig. \ref{fig:numex:div_ph} shows that the divergence of the approximated flux variable $u_h$ indeed approximates the right hand side $-f$ in the $L^2$-norm. 
Note that with the pointwise property $\div u_h = -f_h$ this is merely a measure for the difference between $f^{\mathcal{E}}$ and $f_h$ in the $L^2(\Omega)$-norm, cf. \cref{ass:fh} and \cref{lem:fh:stabl2}.
Furthermore, we observe that the approximation of $p$ without post-processing converges quasi-optimally on $\Omi$, but not on $\Omega$. This shows the necessity of the post-processing step to obtain optimal convergence rates on $\Omega$, which we investigate further in the following section.

\begin{remark}[Hybridization]
  The results that have been obtained for $\gamma_u = 0$ could be reproduced with the hybridized version of the method. Only in some cases we observed consequences of the ill-conditioning of the $A$-block that had an impact when forming the Schur complement during static condensation which seem to have a smaller effect for the original -- also ill-conditioned -- linear system. To remedy the conditioning issues a small amount of regularization of the form $(\varepsilon u_h,v_h)_{\OmG}$ with tiny $\varepsilon > 0$, e.g. $\varepsilon=10^{-10}$ resolved this issue.
\end{remark}

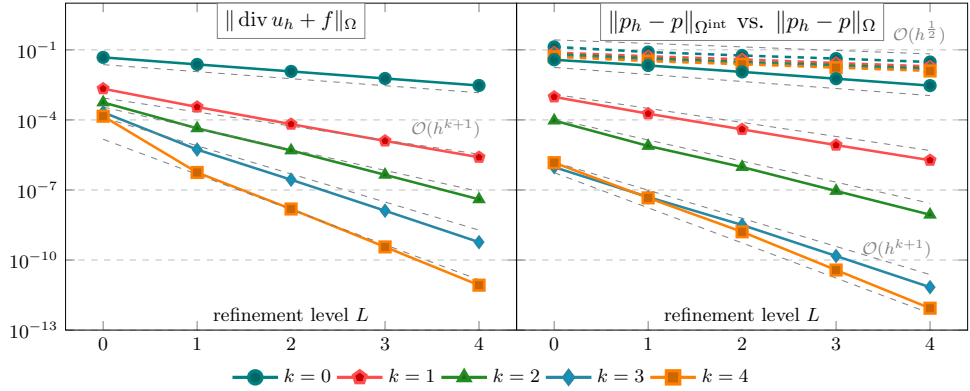
\begin{figure}[!htbp]
  \begin{center}\hspace*{-0.3cm}
      \begin{tikzpicture}[scale=0.8]%
          \begin{groupplot}[%
              group style={%
              group size=2 by 1,
              horizontal sep=0cm,
              vertical sep=0.1cm,
              },
          ymajorgrids=true,
          grid style=dashed,
          ymin=1e-13, ymax=1e1,
          ]       
          \nextgroupplot[width=9cm,height=7cm,domain=0:4,xmode=linear,ymode=log, 
                         x label style={at={(axis description cs:0.5,0.09)},anchor=north,draw=none,fill=white},
                         title style={at={(axis description cs:0.5,0.9625)},anchor=north,draw=gray,fill=white},
                         xlabel={refinement level $L$}, 
                         title={$\Vert \div u_h + f \Vert_{\Omega}$}, 
                         ylabel={}, 
                         cycle list name=unfmixed1, 
              xtick={0,1,2,3,4}, minor tick style = {white}, legend style={legend columns=5, draw=none,nodes={scale=.8}}, 
              legend to name=named
          ]
          \foreach \p in {0,1,2,3,4}{
            \addplot+[discard if not={k}{\p},discard if not={gammastab}{0},discard if not={postprocessver}{0}, line width=\lw] table [x=L, y=udiverror, col sep=comma] {unfmixed_revisited.csv};
          }

          \addplot[gray, dashed, domain=0:4] {(0.025*(1/2^(1))^(x+0.1))};
          \addplot[gray, dashed, domain=0:4] {(0.001*(1/2^(2))^(x+0.1))};
          \addplot[gray, dashed, domain=0:4] {(0.001*(1/2^(3))^(x+0.5))};
          \addplot[gray, dashed, domain=0:4] {(0.001*(1/2^(4))^(x+0.75))};
          \addplot[gray, dashed, domain=0:4] {(0.0001*(1/2^(5))^(x+0.55))};

          \node [draw=none] at (axis description cs:0.85,0.615) {\color{gray}\footnotesize $\!\!\mathcal{O}(h^{k+1})$};           

          \legend{$k=0$, $k=1$, $k=2$, $k=3$, $k=4$}

         \nextgroupplot[width=9cm,height=7cm,domain=0:4,xmode=linear,ymode=log, 
                        x label style={at={(axis description cs:0.5,0.09)},anchor=north,draw=none,fill=white},
                        title style={at={(axis description cs:0.5,0.9625)},anchor=north,draw=gray,fill=white},
                        xlabel={refinement level $L$}, 
                        title={\large $\Vert p_h - p \Vert_{\Omi}$ vs. $\Vert p_h - p \Vert_{\Omega}$}, 
                        ylabel={}, 
                        yticklabels={,,},
                        cycle list name=unfmixed1, 
              legend pos=north east,xtick={0,1,2,3,4}, minor tick style = {white}, legend style={draw=gray, fill=white}
          ]
      
          \foreach \p in {0,1,2,3,4}{
            \addplot+[discard if not={k}{\p},discard if not={gammastab}{0},discard if not={postprocessver}{0}, line width=\lw,dashed] table [x=L, y=pl2error, col sep=comma] {unfmixed_revisited.csv};
          }
          \foreach \p in {0,1,2,3,4}{
            \addplot+[discard if not={k}{\p},discard if not={gammastab}{0},discard if not={postprocessver}{0}, line width=\lw] table [x=L, y=p_inner_l2error, col sep=comma] {unfmixed_revisited.csv};
          }
          
          \addplot[gray, dashed, domain=0:4] {(0.45*(1/2^(1/2))^(x+1.5))};
          \addplot[gray, dashed, domain=0:4] {(0.05*(1/2^(1))^(x+1.5))};
          \addplot[gray, dashed, domain=0:4] {(0.01*(1/2^(2))^(x+1.5))};
          \addplot[gray, dashed, domain=0:4] {(0.0025*(1/2^(3))^(x+1.5))};
          \addplot[gray, dashed, domain=0:4] {(0.0001*(1/2^(4))^(x+1.5))};
          \addplot[gray, dashed, domain=0:4] {(0.0001*(1/2^(5))^(x+1.5))};
          
          \node [draw=none] at (axis description cs:0.9,0.91) {\color{gray}\footnotesize $\!\!\mathcal{O}(h^{\frac12})$};
          \node [draw=none] at (axis description cs:0.85,0.25) {\color{gray}\footnotesize $\!\!\mathcal{O}(h^{k+1})$};           

          \end{groupplot}
      \end{tikzpicture}
      \pgfplotslegendfromname{named}
  \end{center}\vspace*{-0.1cm}
  \caption{$L^2$-convergence of $\div u_h$ towards $-f$ (left) and $L^2$-error of $p_h - p$ on $\Omega$ (dashed) and $\Omi$ (solid) without post-processing (right) for polynomial degrees $k = 0,1,\dots,4$.}
  \label{fig:numex:div_ph}
\end{figure}

\subsection{Post-processings} 
\label{sec:numex:post-processings}

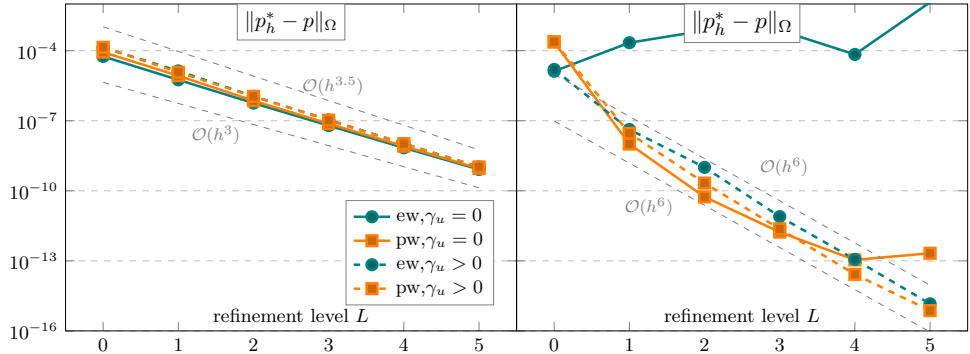
\begin{figure}[!htbp]
    \begin{center}\hspace*{-0.3cm}
        \begin{tikzpicture}[scale=0.8]%
            \begin{groupplot}[%
                group style={%
                group size=2 by 1,
                horizontal sep=0cm,
                vertical sep=0.1cm,
                },
            ymajorgrids=true,
            grid style=dashed,
            ymin=1e-16, ymax=1e-2,
            ]       
            \nextgroupplot[width=9cm,height=7cm,domain=0:5,xmode=linear,ymode=log, 
                           x label style={at={(axis description cs:0.5,0.09)},anchor=north,draw=none,fill=white},
                           title style={at={(axis description cs:0.5,0.9625)},anchor=north,draw=gray,fill=white},
                           xlabel={refinement level $L$}, 
                           title={$\Vert p_h^* - p \Vert_{\Omega}$}, 
                           ylabel={}, 
                           cycle list name=unfmixed2, 
                xtick={0,1,2,3,4,5}, minor tick style = {white}, legend style={at={(axis description cs:0.95,0.4)},draw=gray,fill=white}
            ]

                \addplot+[discard if not={k}{1},discard if not={gammastab}{0},discard if not={postprocessver}{0}, line width=\lw] table [x=L, y=psl2error, col sep=comma] {unfmixed.csv};
                \addplot+[discard if not={k}{1},discard if not={gammastab}{0},discard if not={postprocessver}{1}, line width=\lw] table [x=L, y=psl2error, col sep=comma] {unfmixed.csv};
                \addplot+[discard if not={k}{1},discard if not={gammastab}{\g},discard if not={postprocessver}{0}, line width=\lw,\lsty] table [x=L, y=psl2error, col sep=comma] {unfmixed.csv};
                \addplot+[discard if not={k}{1},discard if not={gammastab}{\g},discard if not={postprocessver}{1}, line width=\lw,\lsty] table [x=L, y=psl2error, col sep=comma] {unfmixed.csv};

                \addplot[gray, dashed, domain=0:5] {(0.0001*(1/2^(1+2))^(x+1.5))};
                \addplot[gray, dashed, domain=0:5] {(0.04*(1/2^(1+2.5))^(x+1.5))};

                \node [draw=none] at (axis description cs:0.6,0.75) {\color{gray}\footnotesize $\!\!\mathcal{O}(h^{3.5})$};           
                \node [draw=none] at (axis description cs:0.34,0.6) {\color{gray}\footnotesize $\!\!\mathcal{O}(h^{3})$};           

            \legend{$\text{ew} {, }\gamma_u = 0 $,$\text{pw} {,} \gamma_u = 0$,$\text{ew} {, } \gamma_u > 0 $,$\text{pw} {,}\gamma_u>0 $}

           \nextgroupplot[width=9cm,height=7cm,domain=0:5,xmode=linear,ymode=log, 
                          x label style={at={(axis description cs:0.5,0.09)},anchor=north,draw=none,fill=white},
                          title style={at={(axis description cs:0.5,0.9625)},anchor=north,draw=gray,fill=white},
                          xlabel={refinement level $L$}, 
                          title={\large $\Vert p_h^* - p \Vert_{\Omega}$}, 
                          ylabel={}, 
                          yticklabels={,,},
                          cycle list name=unfmixed2, 
                legend pos=north east,xtick={0,1,2,3,4,5}, minor tick style = {white}, legend style={draw=gray, fill=white}
            ]
        
                \addplot+[discard if not={k}{4},discard if not={gammastab}{0},discard if not={postprocessver}{0}, line width=\lw] table [x=L, y=psl2error, col sep=comma] {unfmixed.csv};
                \addplot+[discard if not={k}{4},discard if not={gammastab}{0},discard if not={postprocessver}{1}, line width=\lw] table [x=L, y=psl2error, col sep=comma] {unfmixed.csv};
                \addplot+[discard if not={k}{4},discard if not={gammastab}{\g},discard if not={postprocessver}{0}, line width=\lw,\lsty] table [x=L, y=psl2error, col sep=comma] {unfmixed.csv};
                \addplot+[discard if not={k}{4},discard if not={gammastab}{\g},discard if not={postprocessver}{1}, line width=\lw,\lsty] table [x=L, y=psl2error, col sep=comma] {unfmixed.csv};

                \addplot[gray, dashed, domain=0:5] {(0.005*(1/2^(4+2))^(x+1.5))};
                \addplot[gray, dashed, domain=0:5] {(0.00005*(1/2^(4+2))^(x+1.5))};

                \node [draw=none] at (axis description cs:0.6,0.5) {\color{gray}\footnotesize $\!\!\mathcal{O}(h^{6})$};           
                \node [draw=none] at (axis description cs:0.3,0.38) {\color{gray}\footnotesize $\!\!\mathcal{O}(h^{6})$};           

            \end{groupplot}
        \end{tikzpicture}
    \end{center}\vspace*{-0.1cm}
    \caption{Post-processing results for $k=1$ (left) and $k=4$ (right).}
    \label{fig:numex:phs}
\end{figure}

To compare the two post-processing versions introduced in \cref{sec:elpost-process} and \cref{sec:patchpost-process} numerically, we consider numerical results for two polynomial degrees, $k =1$ and $k=4$. Further experiments for other polynomial degrees displayed a similar behavior and are left out. Figure \ref{fig:numex:phs} displays the discretization error of $p$ in the $L^2$-norm on $\dom$ for the elementwise and the patchwise post-processing, with the choices $\gamma_u = 0$ and $\gamma_u = 1$ as above. For the lower order case $k = 1$, both methods exhibit the expected convergence of order (at least) $k+2$, irrespective of whether $\gamma_u = 0$ or $\gamma_u > 0$. 
In the higher order case $k = 4$, both methods perform comparably and converge with the expected order $k +2$ for $\gamma_u > 0$. The patchwise post-processing seems to be slightly more accurate than the elementwise post-processing. In contrast, when no GP is applied, the order of convergence degrades for the elementwise post-processing but remains unaffected for the patchwise post-processing until the ill-conditioning of the unstabilized discretization shows a saturation around a precision of $\approx 10^{-13}$. These results emphasize the requirement of having $\gamma_u > 0$ for the elementwise post-processing. Overall, we observe that the patchwise post-processing is competitive in terms of accuracy and can be applied without requiring GP stabilization.

\subsection{Approximation of $f$} \label{sec:numex:f}
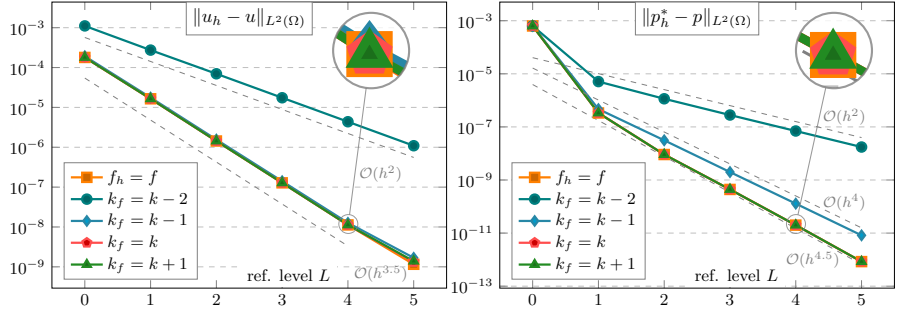
\begin{figure}[!htbp]
  \begin{center}
      \begin{tikzpicture}[scale=0.7,spy using outlines={circle, magnification=4, size=1cm, connect spies}]
          \begin{groupplot}[%
              group style={%
              group size=2 by 1,
              horizontal sep=1cm,
              vertical sep=0.1cm,              
              },
          ymajorgrids=true,
          grid style=dashed,
          ]       
          \nextgroupplot[width=9cm,height=7cm,domain=0:5,xmode=linear,ymode=log, 
                         x label style={at={(axis description cs:0.6,0.09)},anchor=north,draw=none,fill=white},
                         title style={at={(axis description cs:0.5,0.9625)},anchor=north,draw=gray,fill=white},
                         xlabel={ref. level $L$}, 
                         title={$\Vert u_h - u \Vert_{L^2(\Omega)}$}, 
                         ylabel={}, 
                         cycle list name=unfmixed3,                          
              legend pos=south west,xtick={0,1,2,3,4,5},,minor tick style ={white},legend style={draw=gray},
              legend cell align={left}
          ]
          \foreach \f in {0,1,2,3,4}{
              \addplot+[discard if not={k}{2},discard if not={gammastab}{0},discard if not={postprocessver}{0}, discard if not={F}{\f}, line width=\lw] table [x=L, y=ul2error, col sep=comma] {fapprox.csv};
          }
          \addplot[gray, dashed, domain=0:5] {0.003*(1/2^(2))^(x+1.2)};
          \node [draw=none] at (axis description cs:0.84,0.4) {\color{gray}\footnotesize $\!\!\mathcal{O}(h^{2})$};           
          \addplot[gray, dashed, domain=0:4] {0.001*(1/2^(3.5))^(x+1.2)};
\node [draw=none] at (axis description cs:0.84,0.05) {\color{gray}\footnotesize $\!\!\mathcal{O}(h^{3.5})$};           
\legend{$f_h=f$,$k_f=k-2$,$k_f=k-1$,$k_f=k$,$k_f=k+1$}
          \nextgroupplot[width=9cm,height=7cm,domain=0:5,xmode=linear,ymode=log, 
                         x label style={at={(axis description cs:0.6,0.09)},anchor=north,draw=none,fill=white},
                         title style={at={(axis description cs:0.5,0.9625)},anchor=north,draw=gray,fill=white},
                         xlabel={ref. level $L$}, 
                         title={$\Vert p_h^\ast - p \Vert_{L^2(\Omega)}$}, 
                         ylabel={}, 
                         cycle list name=unfmixed3, 
          legend pos=south west,xtick={0,1,2,3,4,5},minor tick style ={white},legend style={draw=gray},
          legend cell align={left}
              ]
                  \foreach \f in {0,1,2,3,4}{
                  \addplot+[discard if not={k}{2},discard if not={gammastab}{0},discard if not={postprocessver}{2}, discard if not={F}{\f},line width=\lw] table [x=L, y=psl2error, col sep=comma] {fapprox.csv};
              }
              \addplot[gray, dashed, domain=0:5] {0.0001*(1/2^(2))^(x+0.65)};
              \node [draw=none] at (axis description cs:0.88,0.6) {\color{gray}\footnotesize $\!\!\mathcal{O}(h^{2})$};           
              \addplot[gray, dashed, domain=0:5] {0.0001*(1/2^(2+2))^(x+0.65)};
              \node [draw=none] at (axis description cs:0.88,0.3) {\color{gray}\footnotesize $\!\!\mathcal{O}(h^{4})$};          
              \addplot[gray, dashed, domain=0:5] {0.00003*(1/2^(2+2.5))^(x+0.65)};
              \node [draw=none] at (axis description cs:0.8,0.1) {\color{gray}\footnotesize $\!\!\mathcal{O}(h^{4.5})$};           
              \legend{$f_h=f$,$k_f=k-2$,$k_f=k-1$,$k_f=k$,$k_f=k+1$}
          \end{groupplot}
          \spy[gray,size=1cm,opacity=.8] on (3.9,0.85) in node [fill=white] at (6,4.5);
          \spy[gray,size=1cm,opacity=.8] on (9.79,0.85) in node [fill=white] at (14.7,4.5);
      \end{tikzpicture}
  \end{center}
  \caption{Error dependence on the approximation of $f_h$ for $k = 2$.}
  \label{fig:numex:fhs}
\end{figure}
Next, we want to investigate the dependency of the discretization error on the approximation of $f$. We restrict to the method with $\gamma_u = 0$ and the patchwise post-processing and set the polynomial order to $k=2$. For the approximation of $f$ we consider the following cases: Either we use \eqref{eq:fh} with $f_h \in \mathbb{P}^{k_f}(\Th)$ and $k_f \in \{k-2,k-1,k,k+1 \}$ or directly choose $f_h = f = \div u$ (on $\OmT$). In \cref{fig:numex:fhs} the corresponding error behaviour for $u-u_h$ and $p - p_h^\ast$ is displayed. We observe that for $\ell \in \{k,k+1\}$ and $f_h = f$ the errors are almost identical and clearly at least as accurate as predicted by our theory (with $\mathcal{O}(h^{k+1.5})$ and $\mathcal{O}(h^{k+2.5})$, respectively, even half an order better). For $k_f = k-2$ the convergence rates drop significantly to $\mathcal{O}(h^{2})$ for both error measures. For $k_f = k-1$ the error in $u$ is essentially unaffected and keeps the convergence rates while for the error in $p$ the convergence rate is precisely four.
Our theory only predicts optimal convergence rates for the case $k_f \geq k$ and for this case the numerical results are in accordance with our expectations. However, as discussed in \cref{rem:fass,rem:freg} it is reasonable to assume that the assumption \cref{ass:fh}, i.e. the dependency on the approximation of $f$, can be slightly weakened which is also suggested by the numerical results for $k_f = k-1$.

\subsection{Neumann boundary conditions}
\noindent With $\gamma_u = 1$ and $\gamma_N^{\text{vol}}(k) = \gamma_N^{\text{facet}}(k) = 0.01$ we consider the discretization in  \eqref{eq:neumann:a} -- \eqref{eq:neumann:c} for the same example as before. We consider only the GP stabilized version here as the formulation in \eqref{eq:neumann:a} -- \eqref{eq:neumann:c} involves not only $u_h$ on $\Omega$ but also $u_h \cdot n$ on the boundary $\Gamma$. To control $u_h \cdot n$ (in the trace sense) with trace inverse inequalities it is hence reasonable to provide control in $L^2(\OmT)$ and not only in $L^2(\Omega)$. Further, we consider the patchwise post-processing only as it does not rely on Dirichlet boundary data. In \cref{fig:uConvRatesNeumann} we observe the corresponding convergence behaviour for $k \in \{0,1,2,3\}$. Optimal convergence rates as in the Dirichlet case can be observed indicating that \eqref{eq:neumann:a} -- \eqref{eq:neumann:c} is indeed a proper discretization for the case of Neumann boundary conditions.
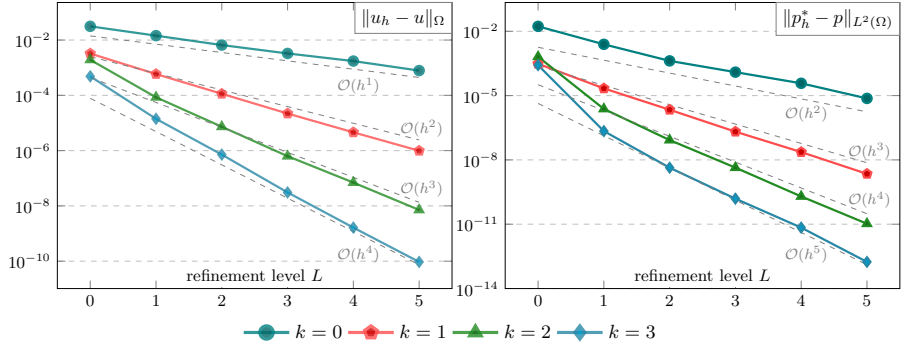
\begin{figure}[!htbp]
  \begin{center}
      \begin{tikzpicture}[scale=0.7]
          \begin{groupplot}[%
              group style={%
              group size=2 by 1,
              horizontal sep=1cm,
              vertical sep=0.1cm,
              },
          ymajorgrids=true,
          grid style=dashed,
          ]       
          \nextgroupplot[width=9cm,height=7cm,domain=0:5,xmode=linear,ymode=log, 
                         xlabel={refinement level $L$},
                         x label style={at={(axis description cs:0.5,0.09)},anchor=north,draw=none,fill=white},
                         title={$\Vert u_h - u \Vert_{\dom}$}, 
                         title style={at={(axis description cs:1,0.9625)},anchor=north east,draw=gray,fill=white},
                         ylabel={}, 
                         cycle list name=unfmixed1, 
                         legend style={legend columns=5, draw=none,nodes={scale=.8}}, 
                         legend to name=named,xtick={0,1,2,3,4,5},
          ]
          \foreach \p in {0,1,2,3,4}{
              \addplot+[discard if not={k}{\p},discard if not={gammastab}{1},discard if not={postprocessver}{0}, line width=\lw,opacity=.8] table [x=L, y=ul2error, col sep=comma] {unfmixed_neumann.csv};
          }
          \foreach \p in {0,1,2,3}{
            \addplot[gray, dashed, domain=0:5] {(0.08*(1/2^(\p+1))^(x+2.5))};
          }          
          \node [draw=none] at (axis description cs:0.77,0.71) {\color{gray}\footnotesize $\!\!\mathcal{O}(h^{1})$};           
          \node [draw=none] at (axis description cs:0.93,0.56) {\color{gray}\footnotesize $\!\!\mathcal{O}(h^{2})$};           
          \node [draw=none] at (axis description cs:0.93,0.35) {\color{gray}\footnotesize $\!\!\mathcal{O}(h^{3})$};           
          \node [draw=none] at (axis description cs:0.77,0.12) {\color{gray}\footnotesize $\!\!\mathcal{O}(h^{4})$};           
      
          \legend{$k=0$,$k=1$,$k=2$,$k=3$}
          \nextgroupplot[width=9cm,height=7cm,domain=0:5,xmode=linear,ymode=log, 
                         x label style={at={(axis description cs:0.5,0.09)},anchor=north,draw=none,fill=white},
                         title={$\Vert p_h^*-p \Vert_{L^2(\dom)}$}, 
                         title style={at={(axis description cs:1,0.9625)},anchor=north east,draw=gray,fill=white},
                         xlabel={refinement level $L$}, 
                         ylabel={}, 
                         cycle list name=unfmixed1,
                         minor tick style ={white},
                         xtick={0,1,2,3,4,5}
              ]
              
              \foreach \p in {0,1,2,3,4}{
                  \addplot+[discard if not={k}{\p},discard if not={gammastab}{1},discard if not={postprocessver}{2}, line width=\lw] table [x=L, y=psl2error, col sep=comma] {unfmixed_neumann.csv};
              }
             \foreach \p in {0,1,2,3}{
                  \addplot[gray, dashed, domain=0:5] {(0.1*(1/2^(\p+2))^(x+2.9))};
              }
              \node [draw=none] at (axis description cs:0.77,0.61) {\color{gray}\footnotesize $\!\!\mathcal{O}(h^{2})$};           
              \node [draw=none] at (axis description cs:0.93,0.48) {\color{gray}\footnotesize $\!\!\mathcal{O}(h^{3})$};           
              \node [draw=none] at (axis description cs:0.93,0.31) {\color{gray}\footnotesize $\!\!\mathcal{O}(h^{4})$};           
              \node [draw=none] at (axis description cs:0.77,0.12) {\color{gray}\footnotesize $\!\!\mathcal{O}(h^{5})$};

          \end{groupplot}
      \end{tikzpicture}
      \pgfplotslegendfromname{named}
  \end{center}
  \caption{Numerical results for Poisson problem with imposed Neumann boundary conditions.} \label{fig:uConvRatesNeumann}
\end{figure}

\subsection{Comparison of stabilizations} \label{sec:compstab}
Within the manuscript we discussed several means to stabilize the unfitted mixed Poisson, respectively the unfitted Darcy problem.
Here, we want to compare and discuss the impact on the computational costs in terms of the sparsity of the resultig system matrix for different approaches.
To discuss the different ghost penalty variants, we define for $\gamma_u,\gamma_{\text{div}},\gamma_p > 0$
\begin{itemize}
  \item $j_{\text{u}}: \Sigma_h \to \Sigma_h^\ast$ to be the operator corresponding to the bilinear form $\Sigma_h \times \Sigma_h \to \mathbb{R}, (u,v) \mapsto \gamma_{u} j_h(u, v)$,
  \item $j_{\text{d}}: \Sigma_h \to Q_h^\ast$ to be the operator corresponding to the bilinear form $\Sigma_h \times Q_h \to \mathbb{R}, (v,q) \mapsto \gamma_{\text{div}} j_h(\div v, q)$, 
  \item $j_{\text{p}}: Q_h \to Q_h^\ast$ to be the operator corresponding to the bilinear form $Q_h \times Q_h \to \mathbb{R}, (p,q) \mapsto \gamma_{p} j_h(p, q)$.
\end{itemize}
These operators allow us to define the five discrete operators displayed in \cref{tab:comparisongp} corresponding to the following five variants:
\begin{enumerate}
  \item[(V1)] First, we consider the restricted mixed formulation as in \eqref{eq:mixedprob} which has the same number of unknowns and couplings as our main formulation (\hyperlink{def:M}{M}) with $\gamma_u=0$.
  \item[(V2)] Next, we consider an additional GP ensuring that the condition number of the $u$-mass matrix is bounded independently of the cut position, i.e. $\gamma_u > 0$ in (\hyperlink{def:M}{M}). 
  \item[(V3)] As a third method, we consider the method from Frachon et al. \cite{frachon2022divergence}, i.e. \eqref{eq:frachmethod} where an additional GP is applied to the divergence operator.
  \item[(V4)] In Puppy \cite{puppi2021cut} and Cao et al. \cite{caoextended} the divergence operator is not stabilized by a GP directly, but a GP is applied as a pressure stabilization coupling $p$ and $q$ (the pressure test and trial functions). %
  \item[(V5)] Finally we consider a worst case scenario where GP is applied on all parts of the discrete operator. This method has not been proposed in the literature so far and should only be seen as an extreme case in this comparison.
\end{enumerate}

\begin{table}
  \begin{center}
\begin{tabular}{|@{}c@{}|@{}c@{}|@{}c@{}|@{}c@{}|@{}c@{}|}
  \hline &&&&\\[-2.5ex]
  V1 & V2 & V3 & V4 & V5 \\[-0.2ex] \hline &&&&\\[-1.5ex]
  $ 
  \left(\!\!\!
  \begin{array}{c@{}c@{}c@{}|@{~}c@{}c@{}c}
    a && & b_{\small\!(\!h\!)}^T &&\\ \hline
    b_{\small\!(\!h\!)} &&&
  \end{array}
  \!\!\!\right)
  $
  & 
$ 
\left(\!\!\!
\begin{array}{c@{}c@{}c@{}|@{~}c@{}c@{}c}
  a &\color{teal!50!black}+& \color{teal!50!black}j_{\text{u}}  & b_{\small\!(\!h\!)}^T &&\\ \hline
  b_{\small\!(\!h\!)} &&& 
\end{array}
\!\!\!\right)
$
&
$ 
\left(\!\!\!
\begin{array}{c@{}c@{}c@{}|@{~}c@{}c@{}c}
  a &\color{teal!50!black}+& \color{teal!50!black}j_{\text{u}}  & b^T &\color{orange!50!black}+&\color{orange!50!black} j_{\text{d}}^T \\ \hline
  b&\color{orange!50!black}+&\color{orange!50!black} j_{\text{d}} &
\end{array}
\!\!\!\right)
$
&
$ 
\left(\!\!\!
\begin{array}{c@{}c@{}c@{}|@{~}c@{}c@{}c}
  a &\color{teal!50!black}+& \color{teal!50!black}j_{\text{u}} & b_{\small\!(\!h\!)}^T&& \\ \hline
  b_{\small\!(\!h\!)}&&  & \color{purple!50!black} \!\!-j_{\text{p}}
\end{array}
\!\!\!\right)
$
&
$ 
\left(\!\!\!
\begin{array}{c@{}c@{}c@{}|@{~}c@{}c@{}c}
  a &\color{teal!50!black}+& \color{teal!50!black}j_{\text{u}} & b^T&\color{orange!50!black}+&\color{orange!50!black}j_{\text{d}}^T \\ \hline 
  b&\color{orange!50!black}+&\color{orange!50!black} j_{\text{d}} &\color{purple!50!black} \!\! -j_{\text{p}}
\end{array}
\!\!\!\right)
$
\\[0.1ex] \hline &&&&
\\[-1.5ex]
\small 
(\hyperlink{def:M}{M}),$\gamma_u\!\!=\! 0$
&
\small 
(\hyperlink{def:M}{M}),$\gamma_u\!\!>\!0$
&
---
&
---
&
---
\\
\small \eqref{eq:mixedprob},\cite{d2012mixed}
&
---
&
\small \eqref{eq:frachmethod}, \cite{frachon2022divergence} 
&
\small\cite{caoextended,puppi2021cut}
&
---
\\
\hline
\end{tabular}
\end{center}
\vspace*{-0.3cm}
\caption{Five different discrete operators with different involvement of the ghost penalty stabilization.}
\label{tab:comparisongp}
\end{table}

For a numerical comparison of these five methods, we consider a specific unit cell representing a part of the unfitted \emph{boundary} of a two-dimensional domain, cf. \cref{fig:unitcell}. 
\begin{figure}
  \begin{tikzpicture}[rotate=0]
    \foreach \y in {1,2,3,4}{
    \draw[fill=gray,opacity=0.2*\y] (\y,0) -- (\y,1) -- (\y+1,0) --cycle;
    \draw[fill=teal,opacity=0.2*\y] (\y,1) -- (\y+1,0) -- (\y+1,1) --cycle;
    \draw[fill=teal,opacity=0.2*\y] (\y,1) -- (\y,2) -- (\y+1,1) --cycle;
    \draw[fill=teal,opacity=0.2*\y] (\y,2) -- (\y+1,1) -- (\y+1,2) --cycle;
    \draw[ultra thick, purple] (\y,1) -- (\y+1,1);
    \draw[ultra thick, purple] (\y,2) -- (\y+1,1);
    }
    \draw[ultra thick, orange] (3, 0) -- (3, 2) -- (3.95, 2) -- (3.95, 0) --cycle;
    \node[right] at (2.8, 2.25) {\color{orange!80!black} unit cell};
    \draw[ultra thick, red] (0.6666, 1.3333) -- (5.333, 1.3333);
    \node[right] at (5.333, 1.3333) {\color{red!80!black} $\partial \Omega$};
    \node[right] at (5.15, 0.5) {\color{gray!80!black} $\downarrow \Omega$};
    \draw[fill=teal,opacity=0.6] (6.5,1) -- (7,1) -- (6.5,1.5) -- cycle;
    \draw[fill=teal,opacity=0.6] (6.5,1) -- (7,1) -- (7,0.5) -- cycle;
    \draw[ultra thick, purple] (6.5,1) -- (7,1);
    \node[right] at (7.1, 1) {\color{purple!80!black} ghost penalty (GP) facet};
    \node[right] at (7.1, 0.5) {\color{teal!80!black} el. connected via GP facet};
   
  \end{tikzpicture}
  \caption{Sketch of the unit cell for the comparison of the different stabilization methods. The unit cell covers two uncut interior elements and two elements cut by the domain boundary. Degrees of freedom (and their induced couplings) associated with the vertical boundary are only counted for the left side.} \label{fig:unitcell}
\end{figure}
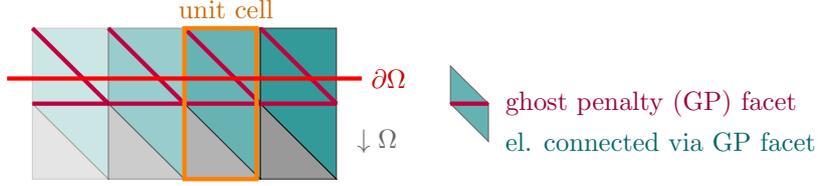
The number of degrees of freedom (\texttt{ndof}) per unit cell is the same for all five methods and we compare only the number of non-zero entries in the system matrix (\texttt{nnz}) per degree of freedom. We do not discuss structural properties of the discretizations which are also affected by the choice of the stabilization method. 
\begin{table}
  \begin{center}
\begin{tabular}{r@{\hspace*{1cm}}rrr@{\hspace*{1cm}}rrrrr}
  \toprule
  $k$ & \multicolumn{3}{c}{\texttt{ndof}} & \multicolumn{5}{c}{\texttt{nnz} per \texttt{dof}} 
  \\
  &$\Sigma_h$&$Q_h$&$\Sigma_h\!\!\times\!\!Q_h$& V1 & V2 & V3 & V4 & V5 \\
  \midrule
0 &   7 &   4 &  11 &   5.00&   6.45 &   7.91 &   7.09 &   8.64 \\
1 &  22 &  12 &  34 &  12.59&  16.82 &  21.06 &  18.68 &  23.18 \\
2 &  45 &  24 &  69 &  22.83&  31.17 &  39.52 &  34.83 &  43.70 \\
3 &  76 &  40 & 116 &  35.72&  49.52 &  63.31 &  55.55 &  70.21 \\
4 & 115 &  60 & 175 &  51.29&  71.86 &  92.43 &  80.86 & 102.71 \\
5 & 162 &  84 & 246 &  69.51&  98.20 & 126.88 & 110.74 & 141.22 \\
6 & 217 & 112 & 329 &  90.40& 128.53 & 166.66 & 145.21 & 185.72 \\
\bottomrule
\end{tabular}
\end{center}
\caption{Comparison of the number of nonzero entries in the system matrix (\texttt{nnz}; with two decimal precision) per degree of freedom (\texttt{dof}) for the five different stabilization methods V1--V5 for different polynomial degrees $k$.} \label{tab:num:comparison}
\end{table}
In \cref{tab:num:comparison} we display the resulting numbers. We observe that the costs of the ghost penalty couplings \emph{on cut elements} are significantly increased by ghost penalty terms with increasing costs from V1 to V2, V4, V3 to V5.  

Note that in an application many elements will be uncut and will have uncut neighbors and the costs associated with these elements will be the same for all the five methods. Here, we isolated only the costs of cut elements. The impact of the differences in the costs depends on factors such as geometry resolution and refinement strategies and will typically be less pronounced than the numbers in \cref{tab:num:comparison} suggest.

\subsection{Conditioning of linear systems}
As a last investigation we want to demonstrate the dependency of the condition number on a stabilization of the bilinear form $a(\cdot,\cdot)$. To this end, we consider the previous setup, but move the center of the geometry through the mesh. We redefine $r(x) = \sqrt{(x_1-x_s)^2+(x_2-x_s)^2}$ where $x_s$ denotes the shifting parameter for a shift in direction $(1,1)$. We fix $k=1$ and consider the coarsest refinement level. For this setup we consider $x_s \in \{0.00001 \cdot z \mid z \in \mathbb{N}_0, z \leq 1000\}$ and compute the condition number of the system matrix for the unstabilized case $\gamma_u=0$ and the GP stabilized case $\gamma_u=1$. We note that both formulations are stable under the assumption of exact arithmetics. However, the conditioning of both methods behaves differently. In \cref{fig:numex:cond} we display the condition number of the system matrix for the two cases. We observe several poles in the condition number for the unstabilized case. These poles are removed by the GP stabilization which guarantees that the conditon numbers are bounded from above by a constant independent of the cut position. 
In \cref{fig:numex:cond} we also show the cut configurations for one case with a moderate condition number and three cut configurations corresponding to three of the poles in the condition number for the unstabilized case. We conclude that the condition number is indeed unbounded for the unstabilized case and ghost penalty stabilization is necessary to ensure a bounded condition number \emph{if no further measures for solving the linear systems are taken}. Let us stress, however, that the unboundedness of the condition number does not imply a stability issue for the scenario of exact arithmetics. It is known that in other cases, see e.g. \cite{L_NM_2017}, specifically tailored preconditioning strategies allowed to solve linear systems robustly and efficiently even if condition numbers got unbounded in bad cut configurations. We leave the investigation of preconditioning strategies for the ($a(\cdot,\cdot)$-)unstabilized unfitted mixed Poisson problem -- especially in view of the hybridized variant -- for future work. 

\begin{figure}[!htbp]
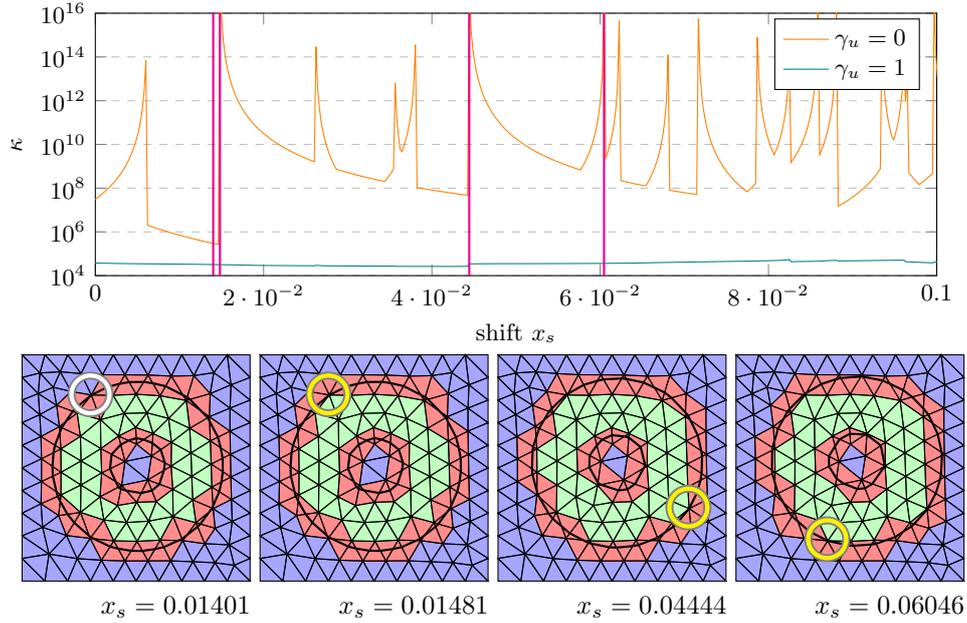

  \begin{center}\hspace*{-0.3cm}
      \begin{tikzpicture}[scale=1.0]
            \begin{axis}[
              xlabel={shift $x_s$},
              ylabel={$\kappa$},
           ymajorgrids=true,
           grid style=dashed,
          xmode=linear, ymode=log, width=\textwidth, height=0.4\textwidth, xmin=0, xmax=0.1, ymin=1e4, ymax=1e16, xtick={0,0.02,0.04,0.06,0.08,0.1}, ytick={1e4,1e6,1e8,1e10,1e12,1e14,1e16},]
              \addplot[discard if not={gammastab}{0}, color=orange] table [x=shift, y=cond, col sep=comma]{unfmixedcond.csv};
              \addplot[discard if not={gammastab}{1}, color=teal] table [x=shift, y=cond, col sep=comma]{unfmixedcond.csv};
              \addplot[thick, samples=50, smooth,domain=0:6,magenta, name path=three] coordinates {(0.01401,1e4)(0.01401,1e16)};
              \addplot[thick, samples=50, smooth,domain=0:6,magenta, name path=three] coordinates {(0.01481,1e4)(0.01481,1e16)};
              \addplot[thick, samples=50, smooth,domain=0:6,magenta, name path=three] coordinates {(0.04444,1e4)(0.04444,1e16)};
              \addplot[thick, samples=50, smooth,domain=0:6,magenta, name path=three] coordinates {(0.06046,1e4)(0.06046,1e16)};
              \legend{$\gamma_u=0$,$\gamma_u=1$}
            \end{axis}
              
        \end{tikzpicture} \\
        \begin{tabular}{r@{~}r@{~}r@{~}r}
        \begin{tikzpicture}[scale=1.5]
          \input{drawmesh_good_1.401e-02.tex}
          \node[draw=gray, ultra thick, circle, minimum size=0.55cm] at (-0.4,0.65){};
          \node[draw=gray, ultra thick, circle, minimum size=0.45cm] at (-0.4,0.65){};
          \node[draw=white, ultra thick, circle, minimum size=0.5cm] at (-0.4,0.65){};
          \end{tikzpicture} & 
        \begin{tikzpicture}[scale=1.5]
        \input{drawmesh_bad_1.481e-02.tex}
        \node[draw=gray, ultra thick, circle, minimum size=0.55cm] at (-0.4,0.65){};
        \node[draw=gray, ultra thick, circle, minimum size=0.45cm] at (-0.4,0.65){};
        \node[draw=yellow, ultra thick, circle, minimum size=0.5cm] at (-0.4,0.65){};
      \end{tikzpicture} &
        \begin{tikzpicture}[scale=1.5]
          \input{drawmesh_bad_4.444e-02.tex}
          \node[draw=gray, ultra thick, circle, minimum size=0.55cm] at (0.67,-0.35){};
          \node[draw=gray, ultra thick, circle, minimum size=0.45cm] at (0.67,-0.35){};
          \node[draw=yellow, ultra thick, circle, minimum size=0.5cm] at (0.67,-0.35){};
        \end{tikzpicture} &
          \begin{tikzpicture}[scale=1.5]
            \input{drawmesh_bad_6.046e-02.tex}
            \node[draw=gray, ultra thick, circle, minimum size=0.55cm] at (-0.2,-0.63){};
            \node[draw=gray, ultra thick, circle, minimum size=0.45cm] at (-0.2,-0.63){};
            \node[draw=yellow, ultra thick, circle, minimum size=0.5cm] at (-0.2,-0.63){};
          \end{tikzpicture} \\
          $x_s=0.01401$ & $x_s=0.01481$ & $x_s=0.04444$ & $x_s=0.06046$ \\
          \end{tabular}
        \end{center}\vspace*{-0.2cm}
  \caption{Condition number depending on cut position and stabilization in a shifted geometry setup for fixed polynomial degree and fixed mesh (top) and four cut configurations related to four specific shifts (before mesh deformation; cf. \cref{sec:geom}). The mesh colors indicate uncut interior elements (green), uncut exterior elements (blue) and cut elements (red). The white and yellow circles indicate the cut changes that correspond to the poles in the condition number.} \label{fig:numex:cond}
\end{figure}

\section*{Data availability} %
\hypertarget{sec:dataavail}{
  The numerical studies in this manuscript can be reproduced and the method can be further investigated with the reproduction data provided in \cite{LYR2CG_2023}.}

\section*{Acknowledgements}
The authors thank Guosheng Fu for re-initiating the search for robust $H(\div)$-conforming discretizations for unfitted domains and for fruitful discussions that lead the authors to the presented developments and investigations.

\nocite{*}
\bibliographystyle{amsplain}
\bibliography{proc}

\appendix
\section{Non-robust inf-sup-stability results for the restricted mixed FEM}\label{sec:rmfem:infsup}
\begin{lemma}\label{lem:rmfem:infsup}
  For the discretization in \cref{sec:restrictedmixedfem} there holds
  \begin{equation}
    \inf_{q_h \in Q_h \setminus \{0\}} \sup_{v_h \in \Shperp\setminus \{0\}} \frac{b(v_h,q_h)}{\Vert v_h \Vert_{H(\div;\Omega)} \Vert q_h \Vert_{\Omega}} > c_R > 0
  \end{equation}
  which implies bijectivity of the operators $b: \Shperp \to Q_h^\ast$ and $b^T: Q_h \to (\Shperp)^\ast$ associated to the bilinear form $b(\cdot,\cdot)$ and hence that \eqref{eq:mixedprob:II}, \eqref{eq:mixedprob:III} and in turn the saddle point problem \eqref{eq:mixedprob} have unique solutions. 
\end{lemma}
\begin{proof}
  It is well-known from mixed FEM theory, cf. e.g. \cite{brezzi2012mixed}, that for all $q_h \in Q_h$ there is $v_h \in \Sh$ so that $\div v_h = q_h$ pointwise. With this choice we obviously obtain the result $b(v_h,q_h)>0$. We note however that due to the unfitted nature of the problem ($\Sh$, $\Shperp$ and $Q_h$ are defined w.r.t. $\OmT$ while the bilinear forms and norms are w.r.t. $\Omega$) the constant $c_R$ depends on the cut configurations and can become arbitrarily small.
\end{proof}

\section{Two ghost penalty stabilization candidates} \label{app:gp}
In the next two sections we briefly present two popular ghost penalty stabilization candidates.
\subsection{Higher order normal derivative jump GP stabilization}
A very established approach for the ghost penalty stabilization is based on higher order normal derivatives, cf. \cite{B15}:
\begin{equation}
    j_F(u_h,v_h) := \sum_{\ell = 0}^k h_F^{2 \ell + 1} ( \jump{\partial_n^\ell u_h}, \jump{\partial_n^\ell u_h} )_F
\end{equation}
where $h_F$ is the local mesh size at facet $F$ and $\jump{\cdot}$ denotes the jump across the facet and $\partial_n^\ell$ is the $\ell$-th order normal derivative.
\subsection{Direct version of the GP stabilization}
An alternative approach, introduced in \cite{preussmaster} avoids forming higher derivatives, but requires volume integrals:
\begin{equation}
    j_F(u_h,v_h) := ((u_h^1-u_h^2),(v_h^1-v_h^2))_{\omega_F}
\end{equation}
where $u_h^i := \mathcal{E}^{\mathcal{P}} u_h {\vert_{T_i}}$, $v_h^i := \mathcal{E}^{\mathcal{P}} v_h {\vert_{T_i}}$, $i=1,2$ where $\mathcal{E}^{\mathcal{P}}$ is the polynomial extension operator; $\omega_F$ is the patch of the two volume elements $T_1$ and $T_2$ adjacent to $F$.

Proofs that \cref{ass:gp} holds for the higher order normal derivative jump GP and the direct GP stabilization are given in (among others) \cite{preussmaster} and \cite[Lemma 5.2]{LO_ESAIM_2019}.

\section{Proof of \cref{lemma:InfSupBh}} \label{app:lbb}
\begin{proof} Slight adaptations to the proof as in e.g. \cite[Section 7.1.2]{brezzi2012mixed} are necessary to account for the fact that the domain $\OmT$ depends on $h$. To this end, we introduce $\Omega_e$, a Lipschitz domain that contains all domains for mesh sizes smaller than an arbitrary fixed mesh size $h_0$, s.t. $\OmT \subset \Omega_e$ for $h \leq h_0$. 
    The divergence operator $b: H(\div;\Omega_e) \to L^2(\Omega_e)$ is surjective. For every $q \in L^2(\Omega_e)$ take the solution to the Cauchy problem $- \Delta \phi = q$ in $\Omega_e$ with $\phi=0$ on $\partial \Omega_e$. Then $u = - \nabla \phi \in H(\div;\Omega_e)$ and $\div u = q$, $\Vert u \Vert_{H(\div;\Omega_e)} \lesssim \Vert q \Vert_{\Omega_e}$. Now, for every $\bar p_h \in Q_h$ set $q= \chi_{\OmT} ~ \bar p_h$ and $u_h = \PiRT u|_{\OmT} \in \Sh$. Then, with the Fortin-property of the Raviart-Thomas interpolator $\PiRT$ we have 
    \begin{align*}
    \div u_h &= \div \PiRT u|_{\OmT} = \Pi^Q (\div u)|_{\OmT} = \Pi^Q q|_{\OmT} = \bar{p}_h \text{ in }\OmT \text{ and }\\[-0.3ex]
    \Vert u_h \Vert_{\Sigma} 
    & \leq 
    \Vert u_h \Vert_{H(\div;\OmT)} 
    \stackrel{(\ast)}{\lesssim}
    \Vert u \Vert_{H(\div;\Omega_e)} 
    \stackrel{(\ast\ast)}{\lesssim}
    \Vert q \Vert_{\Omega_e}
    =
    \Vert q \Vert_{\OmT}
    =
    \Vert \bar{p}_h \Vert_{\OmT}.
    \end{align*}
    Let us stress that $(\ast)$ and $(\ast\ast)$ involve only constants that are independent of $h$ which is due to the Raviart-Thomas interpolation and the domain $\Omega_e$ being independent of $h$.
    Alltogether we have that for every $\bar{p}_h \in Q_h$ there is $u_h \in \Sh$ s.t.
    $$
    b_h(u_h, \bar p_h ) = \Vert \bar p_h \Vert_{\OmT}^2 \gtrsim \Vert \bar{p}_h \Vert_{\OmT} \Vert u_h \Vert_{H(\div,\OmT)} \gtrsim \Vert \bar{p}_h \Vert_{\OmT} \Vert u_h \Vert_{\Sigma}.  %
    $$
\end{proof}

\end{document}